\documentclass[a4paper,12pt,reqno]{amsart}

\usepackage{amssymb,amsmath,array,amscd,hhline}

\usepackage[mathscr]{euscript}
\usepackage{stmaryrd}
\usepackage{ulem}

\usepackage{mathrsfs}

\def\SSigma_#1{\BS(#1)}

\usepackage{tikz-cd}

\def\kappa{\varkappa}

\usepackage{mathrsfs}

\def\({\left(}
\def\){\right)}
\def\[{\left[}
\def\]{\right]}

\def\GL{\mathop{\rm GL}\nolimits}
\def\Gr{\mathop{\rm Gr}\nolimits}

\def\k{\Bbbk}

\def\tr{\mathop{\rm tr}\nolimits}

\def\dlim{\mathop{\rm lim}\limits_{\longrightarrow}}

\def\lm{\lambda}

\def\C{\mathbb C}
\def\R{\mathbb R}

\def\Z{\mathbb Z}

\def\Seq{\mathbf{Seq}}

\def\BS{\mathrm{BS}}

\def\p{{}^p\!}
\def\<{\langle}
\def\>{\rangle}

\def\n{\mathfrak{n}}
\def\r{\mathfrak{r}}

\renewcommand\emptyset{\varnothing}
\renewcommand\phi{\varphi}

\def\im{\mathop{\rm im}}
\def\id{\mathop{\mathrm{id}}}
\renewcommand\epsilon{\varepsilon}

\def\X{\mathcal X}
\def\suchthat{\mathbin{\rm |}}

\def\f{\mathbf f}
\def\ito{\stackrel\sim\to}
\def\pt{{\rm pt}}

\def\q{\mathfrak q}
\def\p{\mathfrak p}

\def\SL{\mathop{\rm SL}\nolimits}
\def\SU{\mathop{\rm SU}\nolimits}

\newtheorem{theorem}{Theorem}
\newtheorem{proposition}[theorem]{Proposition}
\newtheorem{lemma}[theorem]{Lemma}
\newtheorem{example}[theorem]{Example}
\newtheorem{corollary}[theorem]{Corollary}

\newtheorem{definition}[theorem]{Definition}

\def\span{\mathop{\rm span}}

\def\le{\leqslant}
\def\ge{\geqslant}

\def\csh#1#2{\underline{#1}{}_{{}_{\scriptstyle #2}}}

\def\={\equiv}

\def\2{{(2)}}

\title{Nested fibre bundles in Bott-Samelson varieties}
\author{Vladimir Shchigolev}
\address{Financial University under the Government of the Russian Federation\\
49 Leningradsky Prospekt, Moscow, Russia}

\email{shchigolev\_vladimir@yahoo.com}

\begin{document}

\maketitle

\begin{abstract} We give a topological explanation of the main results of
V.\,Shchigolev, Categories of Bott-Samelson Varieties, Algebras and Representation Theory,
{\bf 23} (2), 349--391, 2020. To this end, we consider some subspaces of Bott-Samelson varieties invariant under
the action of the maximal compact torus $K$ and study their topological and
homological properties. Moreover, we describe multiplicative generators
of the equivariant cohomologies of Bott-Samelson varieties.
\end{abstract}

\section{Introduction}
Let $G$ be a semisimple complex algebraic group, $B$ be its Borel subgroup and $T$ be a maximal torus contained in $B$.
For any simple reflection $t$, we have the {\it minimal parabolic subgroup} $P_t=B\cup BtB$.
In this paper, we investigate how Bott-Samelson varieties
$$
\BS(s)=P_{s_1}\times\cdots\times P_{s_n}/B^n
$$
are related
on the levels of topology and cohomology for different sequences of
simple reflections $s=(s_1,\ldots,s_n)$. In the formula above,
the power $B^n$ acts on the product $P_{s_1}\times\cdots\times P_{s_n}$
on the right by the
rule
$$
(g_1,g_2,\ldots,g_n)(b_1,b_2,\ldots,b_n)=(g_1b_1,b_1^{-1}g_2b_2,\ldots,b_{n-1}^{-1}g_nb_n).
$$
Moreover, the torus $T$ acts on $\BS(s)$ via the first factor
$$
t(g_1,g_2,\ldots,g_n)B^n=(tg_1,g_2,\ldots,g_n)B^n.
$$
We have the $T$-equivariant map $\pi(s):\BS(s)\to G/B$ that takes an
orbit $(g_1,g_2,\ldots,g_n)B^n$ to $g_1g_2\cdots g_nB$.

Let $s=(s_1,\ldots,s_n)$ be a subsequence of a sequence $\tilde s=(\tilde s_1,\ldots,\tilde s_{\tilde n})$,
that is, there exists a strictly increasing map $p:\{1,\ldots,n\}\to\{1,\ldots,\tilde n\}$
such that $\tilde s_{p(i)}=s_i$ for any $i=1,\ldots,n$. Then we have a $T$-equivariant embedding
$\BS(s)\hookrightarrow\BS(\tilde s)$
defined by the rule $(g_1,\ldots,g_n)B^n\mapsto(\tilde g_1,\ldots,\tilde g_{\tilde n})B^{\tilde n}$, where
$\tilde g_{p(i)}=g_i$ for any $i=1,\ldots,n$ and $\tilde g_j=1$ for $j$ not in the image of $p$.

Hence we get the restriction map $H^\bullet_T(\BS(\tilde s),\k)\to H^\bullet_T(\BS(s),\k)$
for any ring of coefficients $\k$. It turns out that there exist many other such maps~\cite{CBSV} not necessarily when
$s$ is a subsequence of $\tilde s$. These additional morphisms constitute the so called
{\it folding categories} $\widetilde\Seq$, see~\cite[Section 3.3]{CBSV}.
Until now this was a purely cohomological phenomenon
proved with the help of M.H\"artherich's criterion for the image of the localization~\cite{Haerterich}.

The main aim of this paper is to provide a topological explanation of the
existence of the folding categories and to develop the corresponding topology
(fibre bundles, affine pavings, etc.) that may be interesting in its own right.
Moreover, in a subsequent paper, the author plans to obtain some tensor product theorems for cohomologies of
Bott-Samelson varieties, applying the topological spaces introduced here in full generality.

As was noted in~\cite{CBSV}, the existence of the folding categories can not
be explained on the level of $T$-equivariant topology. The main idea to overcome this obstacle
that we pursue in this paper is to use the equivariant cohomology with respect
to the maximal compact torus $K$. Using this torus and certain compact subgroups $C_{s_i}$
(defined in Section~\ref{Compact subgroups of complex algebraic groups}),
we introduced the {\it compactly defined Bott-Samelson variety}
$$
\BS_c(s)=C_{s_1}\times\cdots\times C_{s_n}/K^n
$$
for any sequence $s=(s_1,\ldots,s_n)$ of reflections not
necessarily simple ones. Here $K^n$ acts on
$C_{s_1}\times\cdots\times C_{s_n}$ on the right by the same
formula as $B^n$ above.
The torus $K$ acts continuously on $\BS_c(s)$ via the first factor and we can consider
the $K$-equivariant cohomology of this space.
Note that topologically $\BS_c(s)$ is nothing more than a Bott tower~\cite{GK}.
However, we consider this space together with the natural projection $\pi_c(s):\BS_c(s)\to C/K$,
where $C$ is the maximal compact subgroup of $G$,
given by $(c_1,\ldots,c_n)K^n\mapsto c_1\cdots c_nK$. This makes the situation more complicated.

From the Iwasawa decomposition it follows that $\BS_c(s)$ and $\BS(s)$
are isomorphic as $K$-spaces if all reflections of $s$ are simple~\cite{GK}. Therefore, we can identify $H^\bullet_K(\BS_c(s),\k)$
and $H^\bullet_T(\BS(s),\k)$.
The space $\BS_c(s)$ has enough $K$-subspaces $\BS_c(s,v)$, see Section~\ref{Nested_structures} for the definitions,
to recover the main results of~\cite{CBSV} by purely topological arguments without resorting to localizations
and M.H\"arterich's criterion. 
The corresponding constructions are
described in Section~\ref{Categories of Bott-Samelson varieties}.

The spaces $\BS_c(s,v)$ are the main technical tool of this paper, where we need only a
very special type of them. The reason why we introduce them
in such a generality is that they will be used latter to prove the tensor product decompositions for cohomologies~\cite{TPDFCOBSV}.
However, the definition of the spaces $\BS_c(s,v)$ looks natural, see~(\ref{eq:bt7:10}), and we would like
to study their structure already here. More exactly, we prove that these spaces can be represented
as twisted products (that is, iterations of fibre bundles, see Theorem~\ref{theorem:1})
of the {\it elementary factors}, which are
the spaces $\BS_c(t)$ and $\pi_c(t)^{-1}(wB)$ for some sequences of reflections $t$ and elements of the Weyl group $w$. 
This decomposition does not allow us to reconstruct $\BS_c(s,v)$
from the elementary factors in a unique way (because of the twist) but in some cases
allows us to compute its $K$-equivarinat cohomology.

One of the cases in which this computation is possible is the case when the space $\BS_c(s,v)$
has a paving by spaces homeomorphic to $\C^d$ for some $d$. We say then that this space has an {\it affine}
paving. In this case, the ordinary odd cohomologies vanish and we can apply spectral
 sequences
for the calculation of $H^\bullet_K(\BS_c(s,v),\k)$. It is still an open question whether
every space $\BS_c(s,v)$ has an affine paving. But we can guarantee this property in one special
case when the pair $(s,v)$ is of gallery type (Definition~\ref{definition:3}).
This notion can be considered as a generalization of the notion of a morphism in the categories
of Bott-Samelson varieties~\cite{CBSV}. It is still unclear why pairs of gallery type
emerge but there are probably plenty of them. For example, if the root system
of $G$ has type $A_1$ or $A_2$, then any pair is of gallery type.

The paper is organized as follows. In Section~\ref{Bott-Samelson varieties}, we
fix the notation for simisimple complex algebraic groups, their compact subgroups
and Bott-Samelson varieties. To avoid renumeration, we consider sequences
as maps defined on totaly ordered sets, see Section~\ref{Sequences and groups}.
In Section~\ref{Nested_fibre_bundles}, we define the space $\BS_c(s,v)$
for a sequence of reflections $s$ and a function $v$ from a nested structure
$R$ (Section~\ref{Nested_structures}) to the Weyl group of $G$. The function $v$ may be considered as a
set of not overlapping restrictions fixing the products over segments.
For example,
\begin{itemize}
\item if $R=\emptyset$, then $\BS_c(s,v)=\BS_c(s)$;
\item if $R=\{(\min I,\max I)\}$, where $s$ is defined on $I$,
      then $\BS_c(s,v)=\pi_c(s)^{-1}(wB)$, where $w$ is the value of $v$;
\item if $R$ consists only of diagonal pairs $(i,i)$, then $\BS_c(s,v)$ is either empty or is isomorphic to $\BS_c(t)$
      for some $t$. This result is used to give a topological prove of the main results of~\cite{CBSV}.
\end{itemize}
Of course, one can combine the above cases. We believe that nested structures
and the spaces $\BS_c(s,v)$ arising from them will be quite useful in establishing
relations between Bott-Samelson varieties. An example of such a relation
is given in Section~\ref{Example}.

Equivariant cohomology is considered in Section~\ref{Equivariant cohomology}. We have different
choices for principal bundles. To get an action of the Weyl group on equivariant cohomologies,
we consider Stiefel manifolds in Section~\ref{Stiefel manifolds}.
Thus we can prove Theorems~1 and~5 from~\cite{CBSV} by purely topological
arguments (Section~\ref{Categories of Bott-Samelson varieties}).

In Section~\ref{Approximation by compact spaces}, we explain how the $K$-equivarinat
cohomology can be computed with the help of compact Stiefel manifolds. We use such approximations
in Section~\ref{Cohomology of BS_c(s)}, where we describe a set of multiplicative generators for $H^\bullet_K(\BS_c(s),\k)$.
If $\k$ is a principal ideal domain with invertible $2$, then this description allows us to prove
in Section~\ref{Copy and concentration operators} by topological arguments
the existence of the copy and concentration operators from~\cite{BTECBSV} and
thus the existence of the corresponding basis elements. 

In this paper, we use only the sheaf cohomology. Therefore, we cite results from~\cite{Iversen} 
and prove some general facts about cohomologies, for example, Lemmas~\ref{lemma:coh:3} and~\ref{lemma:even_bundles}. We also
use some standard notation like $\csh{\k}{X}$ for the constant sheaf on a topological space $X$, $\GL(V)$
for the group of linear isomorphisms $V\to V$, $M_{m,n}(\k)$ the set of $m\times n$-matrices
with entries in $\k$ and $U(n)$ for the group of unitary $n\times n$-matrices. We also apply
spectral sequences associated with fibre bundles.
All of them are first quadrant sequences --- the fact that we tacitly keep in mind.

\section{Bott-Samelson varieties}\label{Bott-Samelson varieties}

\subsection{Compact subgroups of complex algebraic groups}\label{Compact subgroups of complex algebraic groups}
Let $G$ be a semisimple complex group with root system $\Phi$. We assume that $\Phi$ is defined with respect
to the Euclidian space $E$ with the scalar product $(\cdot,\cdot)$. This scalar product
determines a metric on $E$ and therefore a topology. We will denote
by $\overline{A}$ the closure of a subset $A\subset E$.

We consider $G$ as a Chevalley group generated by the root
elements $x_{\alpha}(t)$, where $\alpha\in\Phi$ and $t\in\C$
(see~\cite{Steinberg}). We also fix the following elements of $G$:
$$
w_\alpha(t)=x_\alpha(t)x_{-\alpha}(-t^{-1})x_\alpha(t),\quad
\omega_\alpha=w_\alpha(1),\quad
h_\alpha(t)=w_\alpha(t)\omega_\alpha^{-1}.
$$

For any $\alpha\in\Phi$, we denote by $s_\alpha$ the reflection of $E$
through the hyperplane $L_\alpha$ of the vectors perpendicular to
$\alpha$. We call $L_\alpha$ a {\it wall} perpendicular to
$\alpha$. These reflections generate the Weyl group $W$. We choose
a decomposition $\Phi=\Phi^+\sqcup\Phi^-$ into positive and
negative roots and write $\alpha>0$ (resp. $\alpha<0$) if
$\alpha\in\Phi^+$ (resp. $\alpha\in\Phi^-$). Let
$\Pi\subset\Phi^+$ be the set of simple roots corresponding to
this decomposition. We denote
$$
\mathcal T(W)=\{s_\alpha\suchthat\alpha\in\Phi\},\quad \mathcal
S(W)=\{s_\alpha\suchthat\alpha\in\Pi\}.
$$
and call these sets the {\it set of reflections} and the {\it set
of simple reflections} respectively. We use the standard notation
$\<\alpha,\beta\>=2(\alpha,\beta)/(\beta,\beta)$.

There is the analytic automorphism $\sigma$ of $G$ such that
$\sigma(x_\alpha(t))=x_{-\alpha}(-\bar t)$~\cite[Theorem 16]{Steinberg}.
Here and in what follows $\bar t$ denotes the
complex conjugate of $t$. We denote by $C=G_\sigma$ the set of
fixed points of this automorphism. We also consider the compact
torus $K=T_\sigma=B_\sigma$, where $T$ is the subgroup of $G$
generated by all $h_\alpha(t)$ and $B$ is the subgroup generated
by $T$ and all root elements $x_\alpha(t)$ with $\alpha>0$.

For $\alpha\in\Phi$, we denote by $G_\alpha$ the subgroup of $G$
generated by all root elements $x_\alpha(t)$ and $x_{-\alpha}(t)$
with $t\in\C$. There exists the homomorphism
$\phi_\alpha:\SL_2(\C)\to G$ such that
$$
\(\!
\begin{array}{cc}
1&t\\
0&1
\end{array}\!
\) \mapsto x_\alpha(t),\quad \(\!
\begin{array}{cc}
1&0\\
t&1
\end{array}\!
\) \mapsto x_{-\alpha}(t).
$$
The kernel of this homomorphism is either $\{1\}$ or $\{\pm1\}$.
Let $\sigma'$ be the automorphism of $\SL_2(\C)$ given by
$\sigma'(M)=(\overline M^T)^{-1}$. Here and in what follows, we denote by $M^T$
the transpose of $M$. The we get
$\phi_\alpha\sigma'=\sigma\phi_\alpha$. As the equation
$\sigma'(M)=-M$ does not have a solution, we get
$$
C\cap G_\alpha=\phi_\alpha(\SU_2).
$$
This fact is proved, for example, in~\cite[Lemma 45]{Steinberg}
and is true for any root $\alpha$ not necessarily simple. It
follows from this formula that $\omega_\alpha\in C\cap G_\alpha$.

Let $\mathcal N$ be the subgroup of $G$ generated by the elements
$\omega_\alpha$ and the torus $K$. Clearly $\mathcal N\subset C$.
The arguments of~\cite[Lemma~22]{Steinberg} prove that there
exists an isomorphism $\phi:W\stackrel\sim\to\mathcal N/K$ given
by $s_\alpha\mapsto \omega_\alpha K$.
We choose once and for all a representative $\dot w\in\phi(w)$ for any $w\in W$. 
Abusing notation, we denote $wK=\dot wK=\phi(w)$.


Now consider the product $C_\alpha=\phi_\alpha(\SU_2)K$. It is the
image of the compact space $\phi_\alpha(\SU_2)\times K$ under the
multiplication $G\times G\to G$. Therefore, it is compact and
closed in $G$.
It is easy to check that $C_\alpha$ is a group.
Note that $G_\alpha=G_{-\alpha}$. Therefore $C_\alpha=C_{-\alpha}$
and we can denote $C_{s_\alpha}=C_\alpha$. For any $w\in W$, we have $\dot wC_{s_\alpha}\dot w^{-1}=C_{ws_\alpha w^{-1}}=C_{w\alpha}$.

\subsection{Sequences and products}\label{Sequences and groups}
Let $I$ be a finite totally ordered set. We denote by $\le$ ($<$, $\ge$, etc.) the order on it.
We also add two additional elements $-\infty$ and $+\infty$ with
the natural properties: $-\infty<i$ and $i<+\infty$ for any $i\in
I$ and $-\infty<+\infty$. For any $i\in I$, we denote by $i+1$
(resp. $i-1$) the next (resp. the previous) element of
$I\cup\{-\infty,+\infty\}$. We write $\min I$ and $\max I$ for the
minimal and the maximal elements of $I$ respectively and denote
$I'=I\setminus\{\max I\}$ if $I\ne\emptyset$. Note that
$\min\emptyset=-\infty$ and $\max\emptyset=+\infty$. For $i,j\in
I$, we set $[i,j]=\{k\in I\suchthat i\le k\le j\}$.

Any map $s$ from $I$ to an arbitrary set is called a {\it
sequence} on $I$. We denote by $|s|$ the cardinality of $I$ and call this number the {\it length} of $s$.
For any $i\in I$, we denote by $s_i$ the value of $s$ on $i$.
Finite sequences in the usual sense are sequences on the initial intervals
of the natural numbers, that is, on the sets $\{1,2,\ldots,n\}$.
For any sequence $s$ on a nonempty set $I$, we denote by $s'=s|_{I'}$ its truncation.

We often use the following notation for Cartesian products:
$$
\prod_{i=1}^nX_i=X_1\times X_2\times\cdots\times X_n
$$
and denote by $p_{i_1i_2\cdots i_m}$ the projection of this product to
$X_{i_1}\times X_{i_2}\times\cdots\times X_{i_m}$. If we have maps $f_i:Y\to X_i$ for all
$i=1,\ldots,n$, then we denote by $f_1\boxtimes\cdots\boxtimes f_n$ the map that takes
$y\in Y$ to the $n$-tuple $(f_1(y),\ldots,f_n(y))$. To avoid too many superscripts, we denote
$$
X(I)=\prod_{i\in I}X.
$$
We consider all such products with respect to the product topology if all factors are topological spaces.
Suppose that $X$ is a group. Then $X(I)$ is a topological group with respect to the componentwise
multiplication. In that case, for any $i\in I$ and $x\in X$, we
consider the {\it indicator sequence} $\delta_i(x)\in X(I)$ defined by
$$
\delta_i(x)_j=\left\{
\begin{array}{ll}
x&\text{ if }j=i;\\
1&\text{ otherwise}
\end{array}
\right.
$$

For any $i\in I\cup\{-\infty\}$ and a sequence $\gamma$ on $I$, we use the notation
$\gamma^i=\gamma_{\min I}\gamma_{\min I+1}\cdots\gamma_i$.
Obviously $\gamma^{-\infty}=1$. We set
$\gamma^{\max}=\gamma^{\max I}$.

We will also use the following notation. Let $G$ be a group. If $G$ acts on the set $X$, then we denote by $X/G$
the set of $G$-orbits of elements of $X$. Suppose that $\phi:X\to Y$ is a $G$-equivariant map.
Then $\phi$ maps any $G$-orbit to a $G$-orbit. We denote the corresponding map from $X/G$ to $Y/G$
by $\phi/G$.

\subsection{Galleries}\label{Galleries}
As $L_{-\alpha}=L_\alpha$ for any root $\alpha\in\Phi$, we can denote $L_{s_\alpha}=L_\alpha$.
For any $w\in W$, we get
$wL_{s_\alpha}=wL_\alpha=L_{w\alpha}=L_{s_{w\alpha}}=L_{ws_\alpha
w^{-1}}$.

A {\it chamber} is a connected component of the space
$E\setminus\bigcup_{\alpha\in\Phi}L_\alpha$. We denote by
$\mathbf{Ch}$ the set of chambers and by $\mathbf L$ the set of
walls $L_{s_\alpha}$. We say that a chamber $\Delta$ is {\it
attached} to a wall $L$ if the intersection $\overline\Delta\cap L$ has dimension $\dim E-1$.
The {\it fundamental} chamber is defined by
$$
\Delta_+=\{v\in E\suchthat (v,\alpha)>0\text{ for any
}\alpha\in\Pi\}.
$$
A {\it labelled gallery} on a finite totaly ordered set $I$ is a
pair of maps $(\Delta,\mathcal L)$, where
$\Delta:I\cup\{-\infty\}\to\mathbf{Ch}$ and $\mathcal
L:I\to\mathbf L$ are such that for any $i\in I$,
both chambers $\Delta_{i-1}$ and $\Delta_i$ are attached to
$\mathcal L_i$. The Weyl group $W$ acts on the set of labelled
galleries by the rule $w(\Delta,\mathcal
L)=(w\Delta,w\mathcal L)$, where $(w\Delta)_i=w\Delta_i$ and
$(w\mathcal L)_i=w\mathcal L_i$.

For a sequence $s:I\to\mathcal T(W)$, we define the set of {\it
generalized combinatorial galleries}
$$
\Gamma(s)=\{\gamma:I\to W\suchthat\gamma_i=s_i\text{ or
}\gamma_i=1\text{ for any }i\in I\}.
$$
If $s$ maps $I$ to $\mathcal S(W)$, then the elements of
$\Gamma(s)$ are called {\it combinatorial galleries}.
To any combinatorial galley $\gamma\in\Gamma(s)$ there corresponds
the following labelled gallery:
\begin{equation}\label{eq:5}
i\mapsto \gamma^i\Delta_+,\qquad i\mapsto
L_{\gamma^is_i(\gamma^i)^{-1}}.
\end{equation}
%
Moreover, any labelled gallery $(\Delta,\mathcal L)$ on $I$ such
that $\Delta_{-\infty}=\Delta_+$ can be obtained this way (for the
corresponding $s$ and $\gamma$). If we reflect all chambers $\Delta_j$ and
hyperplanes $\mathcal L_j$ for $j>i$ through the hyperplane $\mathcal L_i$, where $i\in I$,
then we will get a labelled gallery again. This new gallery begins at the same chamber
as the old one. We use the following notation
for this operation on the level of combinatorial galleries: for any $\gamma\in\Gamma(s)$
and index $i\in I$, we denote by $\f_i\gamma$ the combinatorial gallery of $\Gamma(s)$
such that $(\f_i\gamma)_i=\gamma_is_i$ and $(\f_i\gamma)_j=\gamma_j$ for $j\ne i$.

Let $\gamma\in\Gamma(s)$, where $s:I\to\mathcal T(W)$.
%
We consider the sequence $s^{(\gamma)}:I\to\mathcal T(W)$ defined
by
$$
s^{(\gamma)}_i=\gamma^is_i(\gamma^i)^{-1}.
$$
Moreover for any $w\in W$, we define the
sequence $s^w:I\to\mathcal T(W)$ by $(s^w)_i=ws_iw^{-1}$ and the generalized combinatorial gallery
$\gamma^w\in\Gamma(s^w)$ by $\gamma_i^w=w\gamma_iw^{-1}$.

\begin{definition}\label{definition:1}
A sequence $s:I\to\mathcal T(W)$ is of galley type if there exists
a labelled gallery $(\Delta,\mathcal L)$ on $I$ such that
$\mathcal L_i=L_{s_i}$ for any $i\in I$.
\end{definition}

In this case, there exist an element $x\in W$, a sequence of simple
reflections $t:I\to\mathcal S(W)$ and a combinatorial gallery
$\gamma\in\Gamma(t)$ and  such that $x\Delta_{-\infty}=\Delta_+$
and $x(\Delta,\mathcal L)$ corresponds to $\gamma$
by~(\ref{eq:5}). We call $(x,t,\gamma)$ a {\it gallerification} of $s$.
This fact is equivalent to $t^{(\gamma)}=s^x$.

Note that for any $w\in W$ a sequence $s$ is of gallery type if and only if the sequence $s^w$
is of gallery type.

\begin{example}\rm
If $\Phi$ is of type either $A_1$ or $A_2$, then every sequence
$s:I\to\mathcal T(W)$ is of gallery type. In the first case, this
fact is obvious, as there is only one wall.

Let us consider the second case.
We will prove by induction on the cardinality of $I\ne\emptyset$
that for any sequence $s:I\to\mathcal T(W)$ and a chamber
$\Delta_{+\infty}$ attached to $L_{s_{\max I}}$, there exists a
labelled gallery $(\Delta,\mathcal L)$ such that $\mathcal
L_i=L_{s_i}$ for any $i\in I$ and $\Delta_{\max
I}=\Delta_{+\infty}$. The case $|I|=1$ is obvious: we can take,
for example, the labelled gallery $(\Delta,\mathcal L)$ on $I$,
where $\Delta_{-\infty}=\Delta_{\max I}=\Delta_{+\infty}$ and
$\mathcal L_{\max I}=L_{s_{\max I}}$. Now suppose that $|I|>1$ and
the claim is true for smaller sets. If $s_{\max I'}=s_{\max I}$, then we define
$\widetilde\Delta_{+\infty}=\Delta_{+\infty}$. Otherwise let
$\widetilde\Delta_{+\infty}$ be a chamber equal to
$\Delta_{+\infty}$ or obtained from $\Delta_{+\infty}$ by the
reflection through $L_{s_{\max I}}$ such that
$\widetilde\Delta_{+\infty}$ is attached to both $L_{s_{\max I}}$
and $L_{s_{\max I'}}$. The picture below shows how to do it.

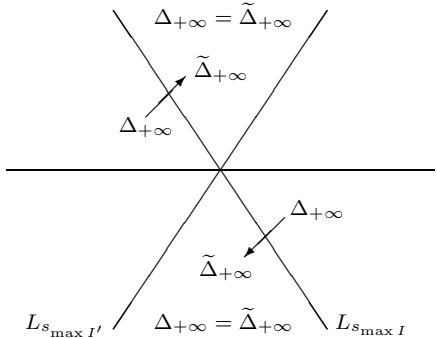
\vspace{52pt}

$$
\begin{picture}(0,0)
\put(-80,0){\line(1,0){160}} \put(0,0){\line(2,3){40}}
\put(0,0){\line(-2,3){40}} \put(0,0){\line(2,-3){40}}
\put(0,0){\line(-2,-3){40}} \put(43,-60){{\tiny$L_{s_{\max I}}$}}
\put(-73,-60){{\tiny $L_{s_{\max I'}}$}} \put(-25,55){{\tiny
$\Delta_{+\infty}=\widetilde\Delta_{+\infty}$}}
\put(-25,-60){{\tiny
$\Delta_{+\infty}=\widetilde\Delta_{+\infty}$}}
\put(-10,35){{\tiny $\widetilde\Delta_{+\infty}$}}
\put(-38,14){{\tiny $\Delta_{+\infty}$}} \put(-8,-40){{\tiny
$\widetilde\Delta_{+\infty}$}} \put(-28,20){\vector(1,1){15}}
\put(26,-17){{\tiny $\Delta_{+\infty}$}}
\put(24,-18){\vector(-1,-1){15}}
\end{picture}
$$

\vspace{70pt}

By the inductive hypothesis, there exists a gallery
$(\Delta,\mathcal L)$ on $I'$ such that $\mathcal L_i=L_{s_i}$ for
any $i\in I'$ and $\Delta_{\max I'}=\widetilde\Delta_{+\infty}$.
It suffices to extend this gallery to $I$ by defining
$\Delta_{\max I}=\Delta_{+\infty}$ and $\mathcal L_{\max
I}=L_{s_{\max I}}$.
\end{example}

\subsection{Definition via the Borel subgroup} For any simple reflection $t\in\mathcal S(W)$,
we consider the minimal parabolic subgroup $P_t=B\cup BtB$. Note
that $C_t$ is a maximal compact subgroups of $P_t$ and
$C_t=P_t\cap C$. Let $s:I\to\mathcal S(W)$ be a sequence of simple
reflections. We consider the space
$$
P(s)=\prod_{i\in I}P_{s_i}
$$
with respect to the product topology. The group $B(I)$ acts on
$P(s)$ on the right by $(pb)_i=b_{i-1}^{-1}p_ib_i$. Here and in what follows, we assume $b_{-\infty}=1$. Let
$$
\BS(s)=P(s)/B(I).
$$
This space is compact and Hausdorff. The Borel subgroup $B$ acts continuously on $P(s)$ by
$$
(bp)_i=\left\{
\begin{array}{ll}
bp_i&\text{ if }i=\min I;\\
p_i&\text{ otherwise.}
\end{array}
\right.
$$
As this action commutes with the right action of $B(I)$ on $P(s)$
described above, we get the left action of $B$ on $\BS(s)$. Note that
$\BS(s)$ is a singleton with the trivial action of $B$ if
$I=\emptyset$.

We also have the map $\pi(s):\BS(s)\to G/B$ that maps $pB(I)$ to $p^{\max}B$.
This map is invariant under the left action of $B$. However, we
need here only the action of the torus $T$. We have
$\BS(s)^T=\{\gamma B(I)\suchthat\gamma\in\Gamma(s)\}$, where
$\gamma$ is identified with the sequence $i\mapsto\dot\gamma_i$ of
$P(s)$. We will identify $\BS(s)^T$ with $\Gamma(s)$. For any
$w\in\Gamma(s)$, we set $\BS(s,w)=\pi(s)^{-1}(wB)$.

\subsection{Definition via the compact torus}\label{Definition via the compact torus}

Let $s:I\to\mathcal T(W)$ be a sequence of reflections. We
consider the space
$$
C(s)=\prod_{i\in I}C_{s_i}
$$
with respect to the product topology. The group $K(I)$ acts on
$C(s)$ on the right by $(ck)_i=k_{i-1}^{-1}c_ik_i$. Here and in what follows, we assume $k_{-\infty}=1$. Let
$$
\BS_c(s)=C(s)/K(I)
$$
and $\nu_c(s):C(s)\to\BS_c(s)$ be the corresponding quotient map.
We denote by $[c]$ the right orbit $cK(I)=\nu_c(s)(c)$ of $c\in C(s)$. 
The space $\BS_c(s)$ is compact and Hausdorff (as the quotient of
a compact Hausdorff space by a continuous action of a compact
group).

The action of $K(I)$ on $C(s)$ described above commutes with the
following left action of $K$:
$$
(kc)_i=\left\{
\begin{array}{ll}
kc_i&\text{ if }i=\min I;\\
c_i&\text{ otherwise.}
\end{array}
\right.
$$
Therefore $K$ acts continuously on $\BS_c(s)$ on the left: $k[c]=[kc]$.

We also have the map $\pi_c(s):\BS_c(s)\to C/K$ that maps
$[c]$ to $c^{\max}K$. This map is obviously invariant under the left action of $K$. We
have $\BS_c(s)^K=\{[\gamma]\suchthat\gamma\in\Gamma(s)\}$, where
$\gamma$ is identified with the sequence $i\mapsto\dot\gamma_i$ of
$C(s)$. We also identify $\BS_c(s)^K$ with $\Gamma(s)$. For any
$w\in W$, we set $\BS_c(s,w)=\pi_c(s)^{-1}(wK)$.
The set of $K$-fixed points of this space is
$$
\BS_c(s,w)^K=\{\gamma\in\Gamma(s)\suchthat\gamma^{\max}=w\}=\Gamma(s,w).
$$

For any $w\in W$, let $l_w:C/K\to C/K$ and $r_w:C/K\to C/K$ be the
maps given by $l_w(cK)=\dot w cK$ and $r_w(cK)=c\dot wK$. They are
obviously well-defined and $l_w^{-1}=l_{w^{-1}}$,
$r_w^{-1}=r_{w^{-1}}$. Moreover, $l_w$ and $r_{w'}$ commute for
any $w,w'\in W$ .
%
%
Let $d_w:\BS_c(s)\to\BS_c(s^w)$ be the map given by
$d_w([c])=[c^w]$, where $(c^w)_i=\dot wc_i\dot w^{-1}$. This map
is well-defined and a homeomorphism. Moreover, $d_w(ka)=\dot
wk\dot w^{-1}a$ for any $k\in K$ and $a\in\BS_c(s)$. We have the
commutative diagram
$$
\begin{tikzcd}[column sep=15ex]
\BS_c(s)\arrow{r}{d_w}[swap]{\sim}\arrow{d}[swap]{\pi_c(s)}&\BS_c(s^w)\arrow{d}{\pi_c(s^w)}\\
C/K\arrow{r}{l_wr_w^{-1}}[swap]{\sim}&C/K
\end{tikzcd}
$$
Hence for any $x\in W$, the map $d_w$ yields the isomorphism
\begin{equation}\label{eq:6}
\BS_c(s,x)\cong\BS(s^w,wxw^{-1}).
\end{equation}

There is a similar construction for a generalized combinatorial gallery
$\gamma\in\Gamma(s)$.
We define the map $D_\gamma:\BS_c(s)\to\BS_c(s^{(\gamma)})$ by
$$
D_\gamma([c])=[c^{(\gamma)}],\text{ where
}c^{(\gamma)}_i=(\dot\gamma_{i-1}\dot\gamma_{i-2}\cdots\dot\gamma_{\min
I})^{-1}c_i\dot\gamma_i\dot\gamma_{i-1}\cdots\dot\gamma_{\min I}.
$$
This map is well-defined. Indeed let $[c]=[\tilde c]$ for some
$c,\tilde c\in C(s)$. Then $\tilde c=ck$ for some $k\in K(I)$.
Then $\tilde c^{(\gamma)}=c^{(\gamma)}k'$, where
$k'_i=(\dot\gamma_i\dot\gamma_{i-1}\cdots\dot\gamma_{\min I})^{-1}k_i\dot\gamma_i\dot\gamma_{i-1}\cdots\dot\gamma_{\min I}$.
The map $D_\gamma$ is obviously continuous and a $K$-equivariant
homeomorphism. We have the following commutative diagram:
$$
\begin{tikzcd}[column sep=15ex]
\BS_c(s)\arrow{r}{D_\gamma}[swap]{\sim}\arrow{d}[swap]{\pi_c(s)}&\BS_c(s^{(\gamma)})\arrow{d}{\pi_c(s^{(\gamma)})}\\
C/K\arrow{r}{r_{\gamma^{\max I}}^{-1}}[swap]{\sim}&C/K
\end{tikzcd}
$$
Hence for any $x\in W$, the map $D_\gamma$ yields the isomorphism
\begin{equation}\label{eq:7}
\BS_c(s,x)\cong\BS_c(s^{(\gamma)},x(\gamma^{\max})^{-1}).
\end{equation}

\subsection{Isomorphism of two constructions}\label{Isomorphism_of_two_constructions}
Consider the Iwasawa decomposition $G=CB$
(see, for example,~\cite[Theorem 16]{Steinberg}). Hence for any
simple reflection $t$, the natural embedding $C_t\subset P_t$
induces the isomorphism of the left cosets $C_t/K\stackrel\sim\to P_t/B$. 
Therefore, for any sequence of simple reflections $s:I\to\mathcal
S(W)$, the natural embedding $C(s)\subset P(s)$ induces the
homeomorphism $\BS_c(s)\stackrel\sim\to\BS(s)$. This homeomorphism
is clearly $K$-equivariant. Moreover, the following diagram is
commutative:
$$
\begin{tikzcd}
\BS_c(s)\arrow{r}{\sim}\arrow{d}&\BS(s)\arrow{d}\\
C/K\arrow{r}{\sim}&G/B
\end{tikzcd}
$$
Note that $\BS_c(s)^K$ is mapped to $\BS(s)^K=\BS(s)^T$. These
sets are identified with $\Gamma(s)$.
%

Any isomorphism of finite totally ordered set $\iota:J\to I$
induces the obvious homeomorphisms
$\BS(s)\stackrel\sim\to\BS(s\iota)$ and
$\BS_c(s)\stackrel\sim\to\BS_c(s\iota)$. They obviously respect
the isomorphisms described above.

\section{Nested fibre bundles}\label{Nested_fibre_bundles}

\subsection{Nested structures}\label{Nested_structures} Let $I$ be a finite totally ordered set.
We denote the order on it by $\le$ (respectively, $\ge$, $<$, etc.).
Let $R$ be a subset of $I^2$. For any $r\in R$, we denote by $r_1$
and $r_2$ the first and the second component of $r$ respectively.
So we have $r=(r_1,r_2)$. We will also use the notation
$[r]=[r_1,r_2]$ for the intervals. We say that $R$ is a {\it
nested structure} on $I$ if
\smallskip
\begin{itemize}
\item $r_1\le r_2$ for any $r\in R$;\\[-8pt]
\item $\{r_1,r_2\}\cap\{r'_1,r'_2\}=\emptyset$ for any $r,r'\in R$;\\[-8pt]
\item for any $r,r'\in R$, the intervals $[r]$ and $[r']$ are
either disjoint or
      one of them is contained in the other.
\end{itemize}
\smallskip
Abusing notation, we will write
\begin{itemize}
\item $r\subset r'$ (resp. $r\subsetneq r'$) to say that $[r]\subset[r']$ (resp. $[r]\subsetneq[r']$);\\[-8pt]
\item $r<r'$ to say that $r_2<r'_1$;
\end{itemize}

Let $s:I\to\mathcal T(W)$ and $v:R\to W$ be arbitrary maps. We
define
\begin{equation}\label{eq:bt7:10}
\BS_c(s,v)=\{[c]\in\BS_c(s)\suchthat\forall r\in R:
c_{r_1}c_{r_1+1}\cdots c_{r_2}\in v_rK\}.
\end{equation}
We consider
this set with respect to the subspace topology induced by its
embedding into $\BS_c(s)$.

Consider the maps $\xi_r:C(s)\to C/K$ defined by
$\xi_r(c)=c_{r_1}c_{r_1+1}\cdots c_{r_2}K$. Then $\BS_c(s,v)$ is
the image of the intersections
$$
C(s,v)=\bigcap_{r\in R}\xi_r^{-1}(v_rK)
$$
under the projection $\nu_c(s):C(s)\to\BS_c(s)$. As $C(s,v)$ is
compact, the space $\BS_c(s,v)$ is also compact and thus is closed
in a Hausdorff space $\BS_c(s)$.

Let $\nu_c(s,v)$ denote the restriction of $\nu_c(s)$ to $C(s,v)$.
We have a Cartesian diagram
$$
\begin{tikzcd}[column sep=15ex]
C(s,v)\arrow{r}{\nu_c(s,v)}\arrow[hook]{d}&\BS_c(s,v)\arrow[hook]{d}\\
C(s)\arrow{r}{\nu_c(s)}&\BS_c(s)
\end{tikzcd}
$$
Clearly the subspace topology on $\BS_c(s,v)$ coincides
with the quotient topology induced by the upper arrow. As
$\nu_c(s)$ is a fibre bundle, $\nu_c(s,v)$ is also a fibre bundle
as a pull-back. The fibres of both bundles are $K(I)$.

\subsection{Projection} Let $F\subset R$ be a nonempty subset such that the intervals $[f]$, where $f\in F$,
are pairwise disjoint. In what follows, we use the notation
$$
F=\{f^1,\ldots,f^\n\}
$$
with elements written in the increasing order:
$f^1<f^2<\cdots<f^\n$.
We set
$$
I^F=I\setminus\bigcup_{f\in F}[f],\quad R^F=R\setminus\{r\in R\suchthat\exists f\in F: r\subset f\}.
$$
Obviously, $R^F$ is a nested structure on $I^F$.

We also introduce the following auxiliary notation. Let $i\in I^F$.
We choose $m=0,\ldots,\n$ so that $f^m_2<i<f^{m+1}_1$.
In these inequalities and in what follows, we assume that $f^0_2=-\infty$ and
$f^{\n+1}_1=+\infty$. We set
$$
v^i=v_{f^1}v_{f^2}\cdots v_{f^m},\quad
\dot v^i=\dot v_{f^1}\dot v_{f^2}\cdots\dot v_{f^m}.
$$
Clearly, $v^i$ is image of $\dot v^i$ under the quotient
homomorphism $\mathcal N\to W$.
We define the sequences $s^F:I^F\to\mathcal T(W)$ and $v^F:R^F\to W$ by
$$
s^F_i=v^is_i(v^i)^{-1},\quad v^F_r=v^{r_1}v_r(v^{r_2})^{-1}.
$$

The main aim of this section is to define the map $p^F:\BS_c(s,v)\to\BS_c(s^F,v^F)$,
which we call the {\it projection along $F$}.

\begin{definition}
A sequence $c\in C(s,v)$ is called $F$-balanced if
\begin{equation}\label{eq:coh:2.5}
c_{f_1}c_{f_1+1}\cdots c_{f_2}=\dot v_f
\end{equation}
for any $f\in F$.
\end{definition}

It is easy to prove that for any $a\in\BS_c(s,v)$, there exists an $F$-balanced $c\in C(s,v)$
such that $a=[c]$ (see the third part of the proof of Lemma~\ref{lemma:3}).
For an $F$-balanced $c\in C(s,v)$, we define the sequence $c^F\in C(s^F)$ by
$$
c^F_i=\dot v^ic_i(\dot v^i)^{-1}.
$$
Then we set $p^F([c])=[c^F]$.

\begin{lemma}\label{lemma:3}
The map $p^F:\BS_c(s,v)\to\BS_c(s^F,v^F)$ is well-defined, $K$-equivariant and continuous.
\end{lemma}
\begin{proof}
{\it Part 1: $p^F$ is well-defined.}
First, we prove that $c^F\in C(s^F,v^F)$ for an $F$-balanced $c$.
Let $r\in R^F$ and $l$ and $m$ be the numbers such that
$$
f^l_2<r_1<f^{l+1}_1, \quad f^m_2<r_2<f^{m+1}_1.
$$
We get
\begin{multline*}
(c^F_{r_1}\cdots c^F_{f^{l+1}_1-1})(c^F_{f^{l+1}_2+1}\cdots c^F_{f^{l+2}_1-1})\cdots
(c^F_{f^m_2+1}\cdots c^F_{r_2})\\
\shoveleft{
=\dot v_{f^1}\dot v_{f^2}\cdots\dot v_{f^l}c_{r_1}\cdots c_{f^{l+1}_1-1}(\dot v_{f^1}\dot v_{f^2}\cdots\dot v_{f^l})^{-1}\times}\\
\times\dot v_{f^1}\dot v_{f^2}\cdots\dot v_{f^{l+1}}
c_{f^{l+1}_2+1}\cdots c_{f^{l+2}_1-1}
(\dot v_{f^1}\dot v_{f^2}\cdots\dot v_{f^{l+1}})^{-1}\times\\
\shoveright{\cdots\times\dot v_{f^1}\dot v_{f^2}\cdots\dot v_{f^m}c_{f^m_2+1}\cdots c_{r_2}(\dot v_{f^1}\dot v_{f^2}\cdots\dot v_{f^m})^{-1}}\\
\shoveleft{=\dot v_{f^1}\dot v_{f^2}\cdots\dot v_{f^l}c_{r_1}\cdots c_{f^{l+1}_1-1}\dot v_{f^{l+1}}c_{f^{l+1}_2+1}\cdots c_{f^{l+2}_1-1}\dot v_{f^{l+2}}\times}\\
\shoveright{\cdots\times\dot v_{f^m}c_{f^m_2+1}\cdots c_{r_2}(\dot v_{f^1}\dot v_{f^2}\cdots\dot v_{f^m})^{-1}}\\
\shoveleft{=\dot v_{f^1}\dot v_{f^2}\cdots\dot v_{f^l}c_{r_1}\cdots c_{r_2}(\dot v_{f^1}\dot v_{f^2}\cdots\dot v_{f^m})^{-1}
=\dot v^{r_1}\dot v_{r}(\dot v^{r_2})^{-1}K\in v^F_rK.}\\
\end{multline*}

Now let us prove that $[c^F]$ is independent of the choice of an $F$-balanced $c$.
Suppose that $[d]=[c]$ for another $F$-balanced sequence $d$.
We get $d=ck$ for some $k\in K(I)$.
Comparing~(\ref{eq:coh:2.5}) and~ the similar equalities for $d$, we get
\begin{equation}\label{eq:coh:7}
\dot v_f k_{f_2}=k_{f_1-1}\dot v_f
\end{equation}
for any $f\in F$. As usually, we assume that $k_{-\infty}=1$.

We define $k^F\in K(I^F)$ by $k^F_i=\dot v^i k_i(\dot v^i)^{-1}$.
We claim that $c^Fk^F=d^F$ and thus $[c^F]=[d^F]$. Indeed, let $i\in I^F$. We choose $m$ so that $f^m_2<i<f^{m+1}_1$.
If $f^m_2<i-1$, then $\dot v^i=\dot v^{i-1}$ and we get
\begin{multline*}
(c^Fk^F)_i=(k^F_{i-1})^{-1}c^F_ik^F_i=(\dot v^ik_{i-1}(\dot v^i)^{-1})^{-1}\dot v^ic_i(\dot v^i)^{-1}\dot v^ik_i(\dot v^i)^{-1}\\
=\dot v^ik_{i-1}^{-1}c_ik_i(\dot v^i)^{-1}=\dot v^id_i(\dot v^i)^{-1}=d^F_i.
\end{multline*}

Now suppose that $i-1=f^m_2$. Let $l=1,\ldots,m$ be the greatest number such that $f^l_1>f^{l-1}_2+1$.
This number is well-defined, as we assumed that $f^0_2=-\infty$. First consider the case $l=1$.
In this case, $i$ is the minimal element of $I^F$. We claim that $k_{f^h_2}=1$ for any $h=0,\ldots,m$.
The case $h=0$ is true by definition. Now suppose that $h<m$ and $k_{f^h_2}=1$. 
By~(\ref{eq:coh:7}), we get
$$
\dot v_{f^{h+1}} k_{f^{h+1}_2}=k_{f^{h+1}_1-1}\dot v_{f^{h+1}}=k_{f^h_2}\dot v_{f^{h+1}}=\dot v_{f^{h+1}}.
$$
Cancelling out $\dot v_{f^{h+1}}$, we get $k_{f^{h+1}_2}=1$. We have proved that $k_{i-1}=k_{f^m_2}=1$.
We get
$$
(c^Fk^F)_i=c^F_ik^F_i=\dot v^ic_i(\dot v^i)^{-1}\dot v^ik_i(\dot v^i)^{-1}=
\dot v^ic_ik_i(\dot v^i)^{-1}=\dot v^ik_{i-1}c_ik_i(\dot v^i)^{-1}=\dot v_id_i\dot v_i^{-1}=d^F_i.
$$

Now consider the case $l>1$. In this case, $j=f^l_1-1$ immediately precedes $i$ in the set $I^F$.
We claim that
\begin{equation}\label{eq:bt6:1}
(\dot v_{f^l}\cdots\dot v_{f^h})^{-1}k_j\dot v_{f^l}\cdots \dot v_{f^h}=k_{f^h_2}
\end{equation}
for any $h=l,\ldots,m$. For $h=l$, this equality directly follows from~(\ref{eq:coh:7}).
Now suppose that $h<m$ and~(\ref{eq:bt6:1}) holds. Conjugating it by $\dot v_{f^{h+1}}^{-1}$,
we get by~(\ref{eq:coh:7}) that
$$
(\dot v_{f^l}\cdots\dot v_{f^{h+1}})^{-1}k_j\dot v_{f^l}\cdots \dot v_{f^{h+1}}
=\dot v_{f^{h+1}}^{-1}k_{f^h_2}\dot v_{f^{h+1}}=\dot v_{f^{h+1}}^{-1}k_{f^{h+1}_1-1}\dot v_{f^{h+1}}
=k_{f^{h+1}_2}.
$$

Applying~(\ref{eq:bt6:1}) for $h=m$, we get
\begin{multline*}
(c^F k^F)_i=\(k^F_j\)^{-1}c^F_ik^F_i
=(\dot v^jk_j(\dot v^j)^{-1})^{-1}\dot v^ic_i(\dot v^i)^{-1}\dot v^ik_i(\dot v^i)^{-1}\\
=\dot v^i((\dot v^j)^{-1}\dot v^i)^{-1}k_j^{-1}((\dot v^j)^{-1}\dot v^i)c_ik_i(\dot v^i)^{-1}
=\dot v^i k_{i-1}^{-1}c_ik_i(\dot v^i)^{-1}=d^F_i.
\end{multline*}

{\it Part 2: $p^F$ is $K$-equivariant.} Let $c\in C(s,v)$ be $F$-balanced.
If $\min I$ belongs to $I^F$, then $kc$ is also $F$-balanced and $(kc)^F=kc^F$ for any $k\in K$.
Thus $p^F(k[c])=[(kc)^F]=k[c^F]=kp^F([c])$ for any $k\in K$.

So we consider the case $\min I\notin I^F$ that is $\min I=f_1^1$.
We can also assume that $I^F\ne\emptyset$, as otherwise $\BS_c(s^F,v^F)$
is a singleton.

Let $m=1,\ldots,\n$ be the smallest number such that $f^m_2+1\in I^F$. This element is obviously the
minimal element of $I^F$. Let
\begin{multline*}
d=k\,c\,\delta_{f^1_2}(\dot v_{f^1}^{-1}k^{-1}\dot v_{f^1})\delta_{f^2_2}((\dot v_{f^1}\dot v_{f^2})^{-1}k^{-1}\dot v_{f^1}\dot v_{f^2})\times\\
\cdots\times\delta_{f^2_m}((\dot v_{f^1}\cdots \dot v_{f^m})^{-1}k^{-1}\dot v_{f^1}\cdots\dot v_{f^m}).
\end{multline*}
Here we multiplied on the right by the indicator sequences on $I$ defined in Section~\ref{Sequences and groups}.
It is easy to check that $d$ is $F$-balanced. Indeed, for any $h=1,\ldots,m$, we get
\begin{multline*}
d_{f^h_1}d_{f^h_1+1}\cdots d_{f^h_2}\\
=
(\dot v_{f^1}\cdots \dot v_{f^{h-1}})^{-1}k\dot v_{f^1}\cdots\dot v_{f^{h-1}}
c_{f^h_1}c_{f^h_1+1}\cdots c_{f^h_2}
(\dot v_{f^1}\cdots \dot v_{f^h})^{-1}k^{-1}\dot v_{f^1}\cdots\dot v_{f^h}\\
=
(\dot v_{f^1}\cdots \dot v_{f^{h-1}})^{-1}k\dot v_{f^1}\cdots\dot v_{f^{h-1}}
\dot v_{f^h}
(\dot v_{f^1}\cdots \dot v_{f^h})^{-1}k^{-1}\dot v_{f^1}\cdots\dot v_{f^h}=\dot v_{f^h}.
\end{multline*}

Therefore $p^F([kc])=[d^F]$.
So it remains to prove that $kc^F=d^F$. Let $i\in I^F$. If $i>f^m_2+1$,
we get $c_i=d_i$ and thus
$$
(kc^F)_i=\dot v_ic_i\dot v_i^{-1}=\dot v^id_i(\dot v^i)^{-1}=d_i^F.
$$
Now let $i=f^m_2+1$. Then $\dot v_i=\dot v_{f^1}\cdots\dot v_{f^m}$ and we get
$$
d^F_i=\dot v^id_i(\dot v^i)^{-1}=\dot v^i(\dot v_{f^1}\cdots\dot v_{f^m})^{-1}k\dot v_{f^1}\cdots\dot v_{f^m}c_i(\dot v^i)^{-1}
=k\dot v^ic_i(\dot v^i)^{-1}=(kc^F)_i.
$$

{\it Part 3: $p^F$ is continuous.} We describe the above mentioned process of bringing an element $c\in C(s,v)$
to its $F$-balanced form in more detail.
For any $m=1,\ldots,n$, let $b_i:C(s,v)\to C(s,v)$ be the following map:
$$
b_m(c)=c\,\delta_{f^m_2}((c_{f^m_1}c_{f^m_1+1}\cdots c_{f^m_2})^{-1}\dot v_{f^m})
$$
This map is continuous and has the following property: if $c$ is $\{f^1,\ldots,f^{m-1}\}$-balanced,
then $b_m(c)$ is $\{f^1,\ldots,f^m\}$-balanced. Thus $b_m\cdots b_1(c)$ is $F$-balanced for
any sequence $c\in C(s,v)$ and moreover $[c]=[b_m\cdots b_1(c)]$.

It remains to consider the continuous map $\tilde p^F(c):C(s,v)\to\BS_c(s^F,v^F)$ given by
$\tilde p^F(c)=[(b_m\cdots b_1(c))^F]$. As is proved before, this maps factors through $\BS_c(s,v)$
and thus yields $p^F$.
\end{proof}

\subsection{Closed nested structures} A nested structure $R$ on $I$ is called {\it closed} if $I\ne\emptyset$ and
it contains the pair $(\min I,\max I)$, which we denote by $\span I$. In this case, we consider the following set
$$
\widetilde C(s,v)=\tilde\xi^{-1}(\dot v_{\span I})\cap\bigcap_{r\in R\setminus\{\span(I)\}}\xi_r^{-1}(v_rK),
$$
where $\tilde\xi:C(s)\to C$ is the map given by $\tilde\xi(c)=c_{\min I}c_{\min I+1}\cdots c_{\max I}$.
We obviously have
$$
\widetilde C(s,v)\subset C(s,v).
$$
The right set is invariant under right action of $K(I)$. However, the left set is not. To get the invariance, let
us consider the embedding $\alpha:K(I')\to K(I)$ given by
$$
\alpha(k)_i=\left\{
\begin{array}{ll}
k_i&\text{ if }i\ne\max I\\
1&\text{ otherwise }
\end{array}
\right.
$$
Then we set
$$
c* k=c\,\alpha(k)
$$
for any $c\in\widetilde C(s,v)$ and $k\in K(I')$.
It is easy to check that $\widetilde C(s,v)$ is invariant under this action of $K(I')$.
Let $\widetilde\nu_c(s,v):\widetilde C(s,v)\to\BS_c(s,v)$ denote the restriction of $\nu_c(s,v)$.

\begin{lemma}\label{lemma:4}
 The map $\widetilde\nu_c(s,v)$ is surjective. Its fibres are exactly the orbits
of the\linebreak $*$-action of $K(I')$. The subspace topology on $\BS_c(s,v)$
induced from its embedding into $\BS_c(s)$ coincides with the quotient topology
induced by $\widetilde\nu_c(s,v)$.
\end{lemma}
\begin{proof} 
Any $K(I')$-orbit is contained in a fibre, as $[c*k]=[c\,\alpha(k)]=[c]$ for any $c\in\widetilde C(s,v)$
and $k\in K(I')$.

Now suppose that $c,\tilde c\in\widetilde C(s,v)$ and $[c]=[\tilde c]$. Then $\tilde c=ck$ for some $k\in K(I)$.
We get
$$
\dot v_{\span I}=\tilde c_{\min I}\tilde c_{\min I+1}\cdots\tilde c_{\max I}
=c_{\min I}c_{\min I+1}\cdots c_{\max I}k_{\max I}=\dot v_{\span I}k_{\max I}.
$$
Hence $k_{\max I}=1$. Thus
$\alpha(k')=k$ and $\tilde c=c\,\alpha(k')=c*k'$.

Finally, let us prove the statement about the surjectivity and the topologies. Let $\iota:\widetilde C(s,v)\to C(s,v)$
denote the natural embedding and $\beta:C(s,v)\to\widetilde C(s,v)$ be the map defined by
$$
\beta(c)=c\,\delta_{\max I}\Big(\tilde\xi(c)^{-1}\dot v_{\span I}\Big).
$$
We get the following commutative diagram:
$$
\begin{tikzcd}[column sep=15ex,row sep=8ex]
\widetilde C(s,v)\arrow{r}{\iota}\arrow{dr}[swap]{\widetilde\nu_c(s,v)}&C(s,v)\arrow{r}{\beta}\arrow{d}{\nu_c(s,v)}&\widetilde C(s,v)\arrow{dl}{\widetilde\nu_c(s,v)}\\
&\BS_c(s,v)&
\end{tikzcd}
$$
The surjectivity of $\widetilde\nu_c(s,v)$ follows from the surjectivity of $\nu_c(s,v)$ and the commutativity
of the right triangle.

For any subset $U\subset\BS_c(s,v)$, we have
$$
(\nu_c(s,v))^{-1}(U)\text{ is open }\Leftrightarrow(\widetilde\nu_c(s,v))^{-1}(U)\text{ is open},
$$
as follows from the continuity of $\iota$ and $\beta$ and the equalities
$$
(\widetilde\nu_c(s,v))^{-1}(U)=\iota^{-1}\nu_c(s,v)^{-1}(U),\quad \nu_c(s,v)^{-1}(U)=\beta^{-1}\widetilde\nu_c(s,v)^{-1}(U).
$$
We have thus proved that the quotient topology on $\BS_c(s,v)$ induced by $\widetilde\nu_c(s,v)$ coincides with
the quotient topology induced by $\nu_c(s,v)$. The latter topology coincides with the subspace topology induced
by the embedding of $\BS_c(s,v)$ to $\BS_c(s)$. 
\end{proof}

{\bf Remark.} It follows from this lemma that $\tilde\nu_c(s,v)$ is an open map.

\subsection{Fibres}\label{Fibres} Let $f\in F$. For any sequence $\gamma$ on $I$, we denote by $\gamma_f$
the restriction of $\gamma$ to $[f]$. We also set
$$
R_f=\{r\in R\suchthat r\subset f\},\quad v_f=v|_{R_f}.
$$
From this construction, it is obvious that $R_f$ is a closed nested structure on $[f]$.

We consider the following continuous map:
$$
\theta:C(s^F,v^F)\times\widetilde C(s_{f^1},v_{f^1})\times\cdots\times\widetilde C(s_{f^{\n}},v_{f^{\n}})
\to C(s)
$$
given by
$$
\theta(c,c^{(1)},\cdots,c^{(\n)})_i=
\left\{
{\arraycolsep=2pt
\begin{array}{ll}
(\dot v^i)^{-1}c_i\dot v^i&\text{ if }i\in I^F;\\[4pt]
c^{(m)}_i&\text{ if }i\in [f^m].
\end{array}}
\right.
$$
We claim that the image of $\theta$ is actually contained in $C(s,v)$. Indeed,
let $r\in R\setminus R^F$. Then there exists $m=1,\ldots,\n$ such that $r\subset f^m$.
We get
$$
\xi_r\theta(c,c^{(1)},\cdots,c^{(\n)})=c^{(m)}_{r_1}c^{(m)}_{r_1+1}\cdots c^{(m)}_{r_2}K=v_rK.
$$
Now we take $r\in R^F$. Let  $l$ and $m$ be the numbers such that
$$
f^l_2<r_1<f^{l+1}_1, \quad f^m_2<r_2<f^{m+1}_1.
$$
We get
\begin{multline*}
\xi_r\theta(c,c^{(1)},\cdots,c^{(\n)})=(\dot v_{f^1}\cdots\dot v_{f^l})^{-1}c_{r_1}\cdots c_{f^{l+1}_1-1}(\dot v_{f^1}\cdots\dot v_{f^l})\times\\
\times c^{(l+1)}_{f^{l+1}_1}c^{(l+1)}_{f^{l+1}_1+1}\cdots c^{(l+1)}_{f^{l+1}_2}(\dot v_{f^1}\cdots\dot v_{f^{l+1}})^{-1}c_{f^{l+1}_2+1}\cdots c_{f^{l+2}_1-1}(\dot v_{f^1}\cdots\dot v_{f^{l+1}})\times\\
\cdots\times c^{(m)}_{f^m_1}c^{(m)}_{f^m_1+1}\cdots c^{(m)}_{f^m_2}(\dot v_{f^1}\cdots\dot v_{f^m})^{-1}c_{f^m_2+1}\cdots c_{r_2}(\dot v_{f^1}\cdots\dot v_{f^m})K\\
=(\dot v_{f^1}\cdots\dot v_{f^l})^{-1}c_{r_1}\cdots c_{f^{l+1}_1-1}c_{f^{l+1}_2+1}\cdots c_{f^{l+2}_1-1}\cdots c_{f^m_2+1}\cdots c_{r_2}(\dot v_{f^1}\cdots\dot v_{f^m})K\\
=(\dot v_{f^1}\cdots\dot v_{f^l})^{-1}v^F_r(\dot v_{f^1}\cdots\dot v_{f^m})K=v_rK.
\end{multline*}

The composition $\nu_c(s)\theta$ obviously factors through the $*$-actions of each group
$K([f^m]')$ on the $m+1$th factor. Thus we get the continuous map
$$
\Theta:C(s^F,v^F)\times\BS_c(s_{f^1},v_{f^1})\times\cdots\times\BS_c(s_{f^{\n}},v_{f^{\n}})
\to\BS_c(s,v)
$$
This map however does not factor in general through the action on
$K(I^F)$
on the first factor.

\begin{lemma}\label{lemma:5}
\begin{enumerate}
\item\label{lemma:5:p:1} $p^F\Theta(c,a^{(1)},\cdots,a^{(\n)})=[c]$.\\[-8pt]
\item\label{lemma:5:p:2} For any $c\in C(s^F,v^F)$, the map $\Theta(c,\_,\ldots,\_)$
maps $\BS_c(s_{f^1},v_{f^1})\times\cdots\times\BS_c(s_{f^{\n}},v_{f^{\n}})$
homeomorphically to $(p^F)^{-1}([c])$.
\end{enumerate}
\end{lemma}
\begin{proof}
\ref{lemma:5:p:1} By Lemma~\ref{lemma:4}, we get for any $m=1,\ldots,\n$ that
$a^{(m)}=[c^{(m)}]$ for some $c^{(m)}\in\widetilde C(s_{f^m},v_{f^m})$.
Note that by definition the sequence $\theta(c,c^{(1)},\cdots,c^{(\n)})$ is $F$-balanced.
Thus
$$
p^F\Theta(c,a^{(1)},\cdots,a^{(\n)})=p^F([\theta(c,c^{(1)},\cdots,c^{(\n)})])=[\theta(c,c^{(1)},\cdots,c^{(\n)})^F]=[c].
$$

\ref{lemma:5:p:2}
As our map is continuous and between compact Hausdorff spaces, it remains to prove that it is bijective.

The first part proves that the image of $\Theta(c,\_,\ldots,\_)$ is indeed contained in $(p^F)^{-1}([c])$.
Let us prove the surjectivity. Let $a\in(p^F)^{-1}([c])$ be an arbitrary point. We can write $a=[d]$ for some $F$-balanced
$d\in C(s,v)$. Then we get $[d^F]=p^F([d])=[c]$. Therefore $d^F k=c$ for some
$k\in K(I^F)$.
Let us define
$k_F\in K(I)$
as follows. We set $(k_F)_i=(\dot v^i)^{-1}k_i\dot v^i$ if $i\in I^F$.
If $i=f^m_2$ for some $m=1,\ldots,\n$, then we denote by $j$
the element of $I^F\cup\{-\infty\}$ immediately preceding $i$.
Then we define
$$
(k_F)_{f^m_2}=(\dot v_{f^1}\cdots\dot v_{f^m})^{-1}k_j\dot v_{f^1}\cdots\dot v_{f^m}
$$
For all other $i$, we set $(k_F)_i=1$.

We claim that $d k_F$ is also $F$-balanced. We get
\begin{multline*}
(dk_F)_{f^m_1}\cdots (dk_F)_{f^m_2}\\
=(\dot v_{f^1}\cdots\dot v_{f^{m-1}})^{-1}k_j^{-1}\dot v_{f^1}\cdots\dot v_{f^{m-1}}d_{f^m_1}\cdots d_{f^m_2}(\dot v_{f^1}\cdots\dot v_{f^m})^{-1}k_j\dot v_{f^1}\cdots\dot v_{f^m}\\
=(\dot v_{f^1}\cdots\dot v_{f^{m-1}})^{-1}k_j^{-1}\dot v_{f^1}\cdots\dot v_{f^{m-1}}\dot v_{f^m}(\dot v_{f^1}\cdots\dot v_{f^m})^{-1}k_j\dot v_{f^1}\cdots\dot v_{f^m}=
\dot v_{f^m}.
\end{multline*}

Let us prove that $(d k_F)_i=(\dot v^i)^{-1}c_i\dot v^i$ for $i\in I^F$. If $i-1\notin\{f^1_2,\ldots,f^m_2\}$,
then we get $\dot v^{i-1}=\dot v^i$ and
\begin{multline*}
(d k_F)_i=(k_F)_{i-1}^{-1}d_i(k_F)_i=(\dot v^{i-1})^{-1}k_{i-1}^{-1}\dot v^{i-1}d_i(\dot v^i)^{-1}k_i^{-1}\dot v_i\\
=(\dot v^i)^{-1}k_{i-1}^{-1}d^F_ik_i^{-1}\dot v^i=(\dot v^i)^{-1}(d^F k)_i\dot v_i=(\dot v^i)^{-1}c_i\dot v^i.
\end{multline*}
If $i=f^m_2+1$ for some $m$, then $\dot v^i=\dot v_{f^1}\cdots\dot v_{f^m}$.
Let $j$ be the element of $I^F\cup\{-\infty\}$ immediately preceding $i$.
We get
\begin{multline*}
(d k_F)_i=(k_F)_{i-1}^{-1}d_i(k_F)_i=
(\dot v_{f^1}\cdots\dot v_{f^m})^{-1}k^{-1}_j\dot v_{f^1}\cdots\dot v_{f^m}d_i(\dot v^i)^{-1}k_i^{-1}\dot v^i\\
=(\dot v^i)^{-1}k^{-1}_j\dot v^id_i(\dot v^i)^{-1}k_i^{-1}\dot v^i=(\dot v^i)^{-1}k^{-1}_jd^F_ik_i^{-1}\dot v_i
=(\dot v^i)^{-1}(d^F k)_i\dot v^i=(\dot v^i)^{-1}c_i\dot v^i.
\end{multline*}
Now it is easy to check that
$$
\theta(c,(d k_F)_{f^1},\ldots,(d k_F)_{f^m})=d k_F.
$$
Hence
$$
\Theta(c,[(d k_F)_{f^1}],\ldots,[(d k_F)_{f^m}])=[d k_F]=[d]=a.
$$

Finally, let us prove the injectivity. Let $a^{(m)},b^{(m)}\in\BS_c(s_{f^m},v_{f^m})$, where $m=1,\ldots,\n$,
be points such that $\Theta(c,a^{(1)},\ldots,a^{(\n)})=\Theta(c,b^{(1)},\ldots,b^{(\n)})$.
We write $a^{(m)}=[c^{(m)}]$ and $b^{(m)}=[d^{(m)}]$ for some sequences
$c^{(m)},d^{(m)}\in\widetilde C(s_{f^m},v_{f^m})$.
Then we get $[\theta(c,c^{(1)},\ldots,c^{(\n)})]=[\theta(c,d^{(1)},\ldots,d^{(\n)})]$.
Thus there exists
$k\in K(I)$
such that
\begin{equation}\label{eq:coh:8}
\theta(c,d^{(1)},\ldots,d^{(\n)})=\theta(c,c^{(1)},\ldots,c^{(\n)}) k.
\end{equation}
We will prove by induction that $k_i=1$ for any $i\in I^F\cup\{f^1_2,\ldots,f^{\n}_2\}\cup\{-\infty\}$.
We start with the obvious case $i=-\infty$. Now suppose that $i\in I^F$ and the claim holds for
smaller indices. We have obviously $i-1\in I^F\cup\{f^1_2,\ldots,f^{\n}_2\}\cup\{-\infty\}$, whence $k_{i-1}=1$.
Evaluating~(\ref{eq:coh:8}) at $i$, we get
$$
(\dot v^i)^{-1}c_i\dot v^i=k_{i-1}^{-1}(\dot v^i)^{-1}c_i\dot v^ik_i=(\dot v^i)^{-1}c_i\dot v^ik_i.
$$
Hence we get $k_i=1$.

Consider the case $i=f^m_2$ for some $m$. Then we have $f^m_1-1\in I^F\cup\{f^1_2,\ldots,f^{\n}_2\}\cup\{-\infty\}$,
whence $k_{f^m_1-1}=1$. We get from~(\ref{eq:coh:8}) that
$$
\dot v_{f^m}=d^{(m)}_{f^m_1}d^{(m)}_{f^m_1+1}\cdots d^{(m)}_{f^m_2}=k^{-1}_{f^m_1-1}c^{(m)}_{f^m_1}c^{(m)}_{f^m_1+1}\cdots c^{(m)}_{f^m_2}k_{f^m_2}=\dot v_{f^m}k_{f^m_2}.
$$
Hence we get $k_{f^m_2}=1$.

From~(\ref{eq:coh:8}), it is obvious now that $d^{(m)}k_{f^m}=c^{(m)}$. Hence $a^{(m)}=b^{(m)}$.
\end{proof}

\subsection{$p^F$ as a fibre bundle} From the results of the previous section, we get the following result.

\begin{theorem}\label{theorem:1}
The map $p^F:\BS_c(s,v)\to\BS_c(s^F,v^F)$ is a fibre bundle with fibre\linebreak
$\BS_c(s_{f^1},v_{f^1})\times\cdots\times\BS_c(s_{f^{\n}},v_{f^{\n}})$.
\end{theorem}
\begin{proof} For brevity, we set $\mathcal D=\BS_c(s_{f^1},v_{f^1})\times\cdots\times\BS_c(s_{f^{\n}},v_{f^{\n}})$.
Let $b\in\BS_c(s^F,v^F)$ be an arbitrary point. As
$$
\nu_c(s^F,v^F):C(s^F,v^F)\to\BS_c(s^F,v^F)
$$
is a fibre bundle, there exists an open neighbourhood $U$ of $b$ and a section $\tau:U\to C(s^F,v^F)$
of this bundle. Now we define the map
$$
\phi:U\times\mathcal D\to\BS_c(s,v)
$$
by
$$
\phi(u,a^{(1)},\ldots,a^{(\n)})=\Theta(\tau(u),a^{(1)},\ldots,a^{(\n)}).
$$
By part~\ref{lemma:5:p:1} of Lemma~\ref{lemma:5}, we get the following commutative diagram:
$$
\begin{tikzcd}
U\times\mathcal D\arrow{rr}{\phi}\arrow{rd}[swap]{p_1}&&\arrow{ld}{p^F}(p^F)^{-1}(U)\\
&U
\end{tikzcd}
$$
Here the right arrow should have been labelled by the restriction $p^F|_{(p^F)^{-1}(U)}$ rather than by the map $p^F$ itself.
We will use similar abbreviations in the sequel.

Part~\ref{lemma:5:p:2} of the same lemma proves that the top arrow is bijective. Let $V$ be an open neighbourhood of $b$
such that $\overline V\subset U$. The the restriction of $\phi$ to $\overline V$ is a homeomorphism
from
$\overline V\times\mathcal D$
to $(p^F)^{-1}(\overline V)$, as both spaces are compact. Then the restriction of $\phi$
to
$V\times\mathcal D$
is a homeomorphism from $V$ to $(p^F)^{-1}(V)$ and the following diagram is commutative:
$$
\begin{tikzcd}
V\times\mathcal D\arrow{rr}{\phi}\arrow{rd}[swap]{p_1}&&\arrow{ld}{p^F}(p^F)^{-1}(V)\\
&V
\end{tikzcd}
$$
\end{proof}

\subsection{Example}\label{Example} Suppose that $G=\SL_5(\C)$ with the following Dynkin diagram:

\begin{center}
\setlength{\unitlength}{1.2mm}
\begin{picture}(45,0)
\put(0,0){\circle{1}}
\put(10,0){\circle{1}}
\put(20,0){\circle{1}}
\put(30,0){\circle{1}}

\put(0.5,0){\line(1,0){9}}
\put(10.5,0){\line(1,0){9}}
\put(20.5,0){\line(1,0){9}}

\put(-1,-3.5){$\scriptstyle\alpha_1$}
\put(9,-3.5){$\scriptstyle\alpha_2$}
\put(19,-3.5){$\scriptstyle\alpha_3$}
\put(29,-3.5){$\scriptstyle\alpha_4$}
\end{picture}
\end{center}

\hspace{20mm}

\noindent
The simple reflections are $\omega_i=s_{\alpha_i}$. Let us consider the following sequence and the element of the Weyl group
$$
s=(\omega_4,   \omega_1,\omega_2,\omega_1,\omega_2,\omega_1  ,\omega_3,\omega_4,\omega_3,\omega_4),\quad w=\omega_2\omega_3\omega_4.
$$
We would like to study the space $\BS_c(s,w)$. Computing directly with matrices of $\SL_5(\C)$,
we can prove that $\BS_c(s,w)=\BS_c(s,v)$, where $v$ is the map from $R=\{(1,10),(2,6)\}$ to $W$ given by
$v((1,10))=w$ and $v((2,6))=\omega_2$. By Theorem~\ref{theorem:1} applied for $F=\{(2,6)\}$,
we get the fibre bundle $p^F:\BS_c(s,w)\to\BS_c(s^F,w\omega_2)$ with fibre
$\BS_c((\omega_1,\omega_2,\omega_1,\omega_2,\omega_1),\omega_2)$. By definition, we get
$s^F=(\omega_4,\omega_2\omega_3\omega_2,\omega_4,\omega_2\omega_3\omega_2,\omega_4)
=(\omega_4,\omega_3,\omega_4,\omega_3,\omega_4)^{\omega_2}$.
Hence by~(\ref{eq:6}), we get $\BS_c(s^F,w\omega_2)\cong\BS_c((\omega_4,\omega_3,\omega_4,\omega_3,\omega_4),\omega_3\omega_4)$.
Replacing compactly defined Bott-Samelson varieties with the usual ones as described in
Section~\ref{Isomorphism_of_two_constructions}, we get the following fibre bundle:
$$
\begin{tikzcd}
\BS((\omega_1,\omega_2,\omega_1,\omega_2,\omega_1),\omega_2)\arrow{r}&
\BS(s,w)\arrow{r}{p^F}&\BS((\omega_4,\omega_3,\omega_4,\omega_3,\omega_4),\omega_3\omega_4).
\end{tikzcd}
$$
By~\cite[Corollary 5]{CBSV}, all three spaces above are smooth.

\subsection{Affine pavings}\label{Affine_pavings} We say that a topological space $X$ has an {\it affine paving} if there exists a filtration
$$
\emptyset=X_0\subset X_1\subset X_2\subset\cdots\subset X_{n-1}\subset X_n=X
$$
such that all $X_i$ are closed and each difference $X_i\setminus X_{i-1}$ is homeomorphic to $\C^{m_i}$
for some integer $m_i$. As any fibre bundle over a Euclidian space is trivial,
we get the following result.

\begin{proposition}\label{proposition:1}
Let $Y\to X$ be a fibre bundle with fibre $Z$. If $X$ and $Z$ have affine pavings, then $Y$ also has.
\end{proposition}

For the next lemma, remember Definition~\ref{definition:1}.

\begin{lemma}\label{lemma:8}
Let $s:I\to\mathcal T(W)$ be a sequence of gallery type. For any $w\in W$, the space
$\BS_c(s,w)$
has an affine paving.
\end{lemma}
\begin{proof}
%
Let $(x,t,\gamma)$ be a gallerification of $s$. We have $t^{(\gamma)}=s^x$.
Applying~(\ref{eq:6}) and~(\ref{eq:7}), we get the homeomorphisms
$$
\BS_c(s,w)\cong\BS_c(s^x,xwx^{-1})=\BS_c(t^{(\gamma)},xwx^{-1})\cong\BS_c(t,xwx^{-1}\gamma^{\max})\cong\BS(t,xwx^{-1}\gamma^{\max}).
$$
The last space has an affine paving by~\cite[Proposition 2.1]{Haerterich}.
\end{proof}

Let $r\in R$ and $r^1,\ldots,r^\mathfrak q$ be the maximal (with respect to the inclusion)
elements of those elements of $R$ that are strictly contained in $r$. We write them in the increasing order
$r^1<\cdots<r^\mathfrak q$.
We set
$$
I(r,R)=[r]\setminus\bigcup_{r'\in R\atop r'\subsetneq r}[r']=[r]\setminus\bigcup_{m=1}^{\mathfrak q}[r^m].
$$
We define the sequence  $s^{(r,v)}:I(r,R)\to\mathcal T(W)$ by
$$
s^{(r,v)}_i=v_{r^1}\cdots v_{r^m}s_i(v_{r^1}\cdots v_{r^m})^{-1},
$$
where $r_2^m<i<r_1^{m+1}$. Here we suppose as usually that $r_2^0=-\infty$ and $r_1^{\q+1}=+\infty$.

\begin{definition}\label{definition:3}
Let $s$ be a sequence of reflections on $I$ and $v$ be a map from a nested structure $R$ on $I$
to $W$. We say that the pair $(s,v)$ is {\it of gallery type} if for any $r\in R$,
the sequence $s^{(r,v)}$ is of gallery type.
\end{definition}
Note that if $R=\emptyset$, then $(s,v)$ is always of gallery type, and if $R=\{\span I\}$,
then $(s,v)$ is of gallery type if and only if $s$ is so.

\begin{lemma}\label{lemma:7}
If $(s,v)$ is of gallery type, then $(s^F,v^F)$ and $(s_{f^1},v_{f^1}),\ldots,(s_{f^{\n}},v_{f^{\n}})$
are also of gallery type.
\end{lemma}
\begin{proof}First let us prove that $(s^F,v^F)$ is of gallery type. Let $r\in R^F$, the pairs $r^1,\ldots,r^\mathfrak q$
be chosen as above and $l$ and $m$ be numbers such that
$$
f^l_2<r_1<f^{l+1}_1, \quad f^m_2<r_2<f^{m+1}_1.
$$

Let $r^{i_1},\ldots,r^{i_{\mathfrak p}}$, where $i_1<\cdots<i_{\mathfrak p}$,
be those elements of $r^{1},\ldots,r^{\mathfrak q}$ that do not belong to $F$.
They
are all maximal elements of $R^F$ strictly contained in $r$.
It is convenient to set $i_0=0$ and $i_{\p+1}=\q+1$.

For any $t=1,\ldots,\mathfrak p$,
we choose the numbers $a_t$ and $b_t$ so that
$$
f^{a_t}_2<r_1^{i_t}<f^{a_t+1}_1,\qquad f^{b_t}_2<r_2^{i_t}<f^{b_t+1}_1.
$$
Note that $I(r,R)=I^F(r,R^F)$ and
$$
(f^{l+1},\ldots,f^{a_1},r^{i_1},f^{b_1+1},\ldots,f^{a_2},r^{i_2},f^{b_2+1},\ldots,r^{(i_{\mathfrak p})},f^{b_{\mathfrak p}+1},\ldots,f^m)
=(r^{1},\ldots,r^{\mathfrak q}),
$$
where $a_1=m$ if $\p=0$. 
Now we take any $i\in I^F(r,R^F)$. First choose $h=0,\ldots,\p$ so that $r^{i_h}_2<i<r^{i_{h+1}}_1$.
Then let $j=0,\ldots,\n$ be such that $f^j_2<i<f^{j+1}_1$.
Note that $b_h\le j\le a_{h+1}$. We get
$$
(s^F)^{(r,v^F)}_i=v^F_{r^{i_1}}\cdots v^F_{r^{i_h}}s^F_i(v^F_{r^{i_1}}\cdots v^F_{r^{i_h}})^{-1}.
$$
We can compute the needed ingredients of this formula as follows:
\begin{multline*}
v^F_{r^{i_1}}\cdots v^F_{r^{i_h}}=v_{f^1}\cdots v_{f^{a_1}}v_{r^{i_1}}(v_{f^1}\cdots v_{f^{b_1}})^{-1}
v_{f^1}\cdots v_{f^{a_2}}v_{r^{i_2}}(v_{f^1}\cdots v_{f^{b_2}})^{-1}\times\\
\cdots v_{f^1}\cdots v_{f^{a_h}}v_{r^{i_h}}(v_{f^1}\cdots v_{f^{b_h}})^{-1}\\
=v_{f^1}\cdots v_{f^{a_1}}v_{r^{i_1}}v_{f^{b_1+1}}\cdots v_{f^{a_2}}v_{r^{i_2}}\cdots
v_{r^{i_h}}(v_{f^1}\cdots v_{f^{b_h}})^{-1},\\
\shoveleft{
s^F_i=v_{f^1}\cdots v_{f^j}s_i(v_{f^1}\cdots v_{f^j})^{-1}.}\\
\end{multline*}
As a result, we get
$$
(s^F)^{(r,v^F)}_i=v_{f^1}\cdots v_{f^l}s^{(r,v)}_i(v_{f^1}\cdots v_{f^l})^{-1}.
$$
This proves that $(s^F)^{(r,v^F)}$ is of gallery type.

Now let us prove that each $(s_{f^m},v_{f^m})$ is of gallery type. Note that for any $r\in R_{f^m}$
the pairs $r^{1},\ldots,r^{\mathfrak q}$ are all maximal pairs of $R_{f^m}$ strictly contained in $r$ and
$[f^m](r,R_{f^m})=I(r,R)$.
Hence we get $(s_{f^m})^{(r,v_{f^m})}=s^{(r,v)}$ and this sequence is of gallery type.
\end{proof}

\begin{corollary}\label{corollary:1}
If $(s,v)$ is of gallery type, then $\BS_c(s,v)$ has an affine paving.
\end{corollary}
\begin{proof}
Arguing by induction on the cardinality of $R$, applying Theorem~\ref{theorem:1}, Proposition~\ref{proposition:1} and
Lemma~\ref{lemma:7}, it suffices to consider the cases $R=\emptyset$
and $R=\{\span I\}$. It the first case $\BS_c(s,v)=\BS_c(s)$. This space has an affine paving
as a Bott tower~\cite[Proposition 3.10]{GK}.
In the second case, we have
$\BS_c(s,v)=\BS_c(s,v_{\span I})$.
This space has an affine paving by Lemma~\ref{lemma:8}.
\end{proof}


\section{Equivariant cohomology and categories of Bott-Samelson varieties}\label{Equivariant cohomology}

\subsection{Definitions}\label{Def_equiv}
Let
$p_T:E_T\to B_T$
be a universal principal $T$-bundle. For any $T$-space $X$, we consider the {\it Borel construction}
$X\times_TE_T=(X\times E_T)/T$, where $T$ acts on the Cartesian product diagonally: $t(x,e)=(tx,te)$.
Then we define the {\it $T$-equivariant cohomology of $X$ with coefficients $\k$} by
$$
H^\bullet_T(X,\k)=H^\bullet(X\times_TE_T,\k).
$$
This definition a priory depends on the choice of a universal principal $T$-bundle.
However, all these cohomologies are isomorphic. Indeed let
$p'_T:E'_T\to B'_T$
be another universal principal $T$-bundle. We consider the following diagram (see~\cite[1.4]{Jantzen}):
\begin{equation}\label{eq:bt6:2.5}
\begin{tikzcd}
{}&(X\times E_T\times E'_T)/T\arrow{rd}{p_{13}/T}\arrow{ld}[swap]{p_{12}/T}&\\
X\times_TE_T&&X\times_TE'_T
\end{tikzcd}
\end{equation}
where the $T$ acts diagonally on $X\times E_T\times E'_T$. 
As $p_{12}/T$ and $p_{13}/T$ are fibre bundles with fibres $E'_T$
and $E_T$ respectively, we get by the Vietoris-Begle mapping theorem (see Lemma~\ref{lemma:coh:3})
the following diagram for cohomologies:
\begin{equation}\label{eq:bt6:2}
\begin{tikzcd}
{}&H^\bullet((X\times E_T\times E'_T)/T,\k)&\\
H^\bullet(X\times_TE_T,\k)\arrow{ru}[swap]{\sim}{(p_{12}/T)^*}&&H^\bullet(X\times_TE'_T,\k)\arrow{lu}{\sim}[swap]{(p_{13}/T)^*}
\end{tikzcd}
\end{equation}
It allows us to identify the cohomologies $H^\bullet(X\times_TE_T,\k)$ and $H^\bullet(X\times_TE'_T,\k)$.
The reader can easily check that this identification respects composition and
is identical if both universal principal bundles are equal.
%

Similar constructions are possible for the compact torus $K$: for a $K$-space $X$, we define
$$
H^\bullet_K(X,\k)=H^\bullet(X\times_KE_K,\k),
$$
where
$p_K:E_K\to B_K$
is a universal principal $K$-bundle. There is the following choice of universal principal
$T$- and $K$-bundles that allows us to identify $T$- and $K$-equivariant cohomologies.
Let $T\cong(\C^\times)^d$ for the corresponding $d$. Then $K\cong(S^1)^d$. We consider the space
$\mathcal E_T=(\C^\infty\setminus\{0\})^d$
and its subspace $\mathcal E_K=(S^\infty)^d$, where
$$
\C^\infty=\dlim\C^n,\quad S^\infty=\dlim S^n
$$
($S^n$ denotes the $n$-dimensional sphere). Let us also consider the space
$$
\mathcal B=(\C P^\infty)^d=\dlim(\C P^n)^d
$$
together with the natural maps $\mathcal E_T\to B$ and $\mathcal E_K\to\mathcal B$.
They
are universal principal $T$- and $K$-bundles respectively.
For each $T$-space $X$, the spaces $X\times_T\mathcal E_T$ and $X\times_K\mathcal E_K$ are homeomorphic~\cite[1.6]{Jantzen}.
Hence $H^\bullet_T(X,\k)=H^\bullet_K(X,\k)$.


\subsection{Stiefel manifolds}\label{Stiefel manifolds} The universal principal $K$-bundle $\mathcal E_K\to\mathcal B$ considered above is a classical
choice for calculating $K$-equivariant cohomologies. However, it has the following disadvantage: we do not know how to extend
the action of $K$ on $\mathcal E_K$ to a continuous action of the maximal compact subgroup $C$.

The solution to this problem is to consider an embedding of $C$ into a unitary group and
take a universal principal for this group. We explain this construction in a little more detail.

Being a Chevalley group, the group $G$ admits an embedding $G\le\GL(V)$
for some faithful representation $V$ of the Lie algebra of $G$. Let $V=V_1\oplus\cdots\oplus V_k$
be a decomposition into a direct sum of irreducible $G$-modules. By~\cite[Chapter 12]{Steinberg}, there exist for each $i$ a positive
definite Hermitian form $\<,\>_{V_i}$ on $V_i$ such that $\<gu,v\>_{V_i}=\<u,\sigma(g^{-1})v\>_{V_i}$ for any $x\in G$
and $u,v\in V_i$, where $\sigma$ is the automorphism of $G$ defined in Section~\ref{Compact subgroups of complex algebraic groups}.
Their direct sum $\<,\>_{V}$ is a positive definite Hermitian form on $V$ satisfying the same property.
We consider the unitary group
$$
U(V)=\{g\in\GL(V)\suchthat\<gu,gv\>_{V}=\<u,v\>_V\;\forall u,v\in{V}\},
$$
for which we have $U(V)\cap G=C$. Choosing an orthonormal basis $(v_1,\ldots,v_{\mathfrak r})$ of $V$,
we get an isomorphism $U(V)\stackrel\sim\to U(\mathfrak r)$ taking an operator of $U(V)$ to its matrix in this basis.

For any natural number $N\ge\r$,
we consider the {\it Stiefel manifold}\footnote{This space is usually denoted by $V_\r(\C^N)$ or $\C V_{N,\r}$.
We also transpose matrices, as we want to have a left action of $U(\r)$.}
$$
E^N=\{A\in M_{\r,N}(\C)\suchthat A\bar A^T=I_\r\},
$$
where $M_{\r,N}(\C)$ is the space of $\r\times N$ matrices with respect to metric topology and $I_\r$ is the identity
matrix. The group $U(\r)$ of unitary $\r\times \r$ matrices acts on $E^N$ on the left by multiplication.
Similarly, $U(N)$ acts on $E^N$ on the right. The last action is transitive and both actions commute.
The quotient space $\Gr^N=E^N/U(\r)$ is called a {\it Grassmanian} and the corresponding
quotient map $E^N\to\Gr^N$ is a principal $U(\r)$-bun\-dle. Note that the group $U(N)$ also acts on $\Gr^N$
by the right multiplication. For $N'>N$, we get the embedding $E^N\hookrightarrow E^{N'}$ by adding $N'-N$
zero columns to the right.

Taking the direct limits
$$
E^\infty=\dlim E^N,\quad \Gr=\dlim\Gr^N,
$$
we get a universal principal $U(\r)$-bundle $E^\infty\to\Gr$.

We need the spaces $E^N$ to get the principal $K$-bundles $E^N\to E^N/K$.
It is easy to note that this bundle for $N=\infty$ is the direct limit of the bundles for $N<\infty$.

\subsection{Equivariant cohomology of a point}\label{Equivariant_cohomology_of_a_point}
We denote by $S=H^\bullet_T(\pt,\k)$ the equivariant cohomology of a point.
It is well known that $S$ is a polynomial ring with zero odd degree component. More exactly, let $\mathfrak X(T)$ be the group of all
continuous homomorphisms $T\to\C^\times$. For each $\lm\in\mathfrak X(T)$, let $\C_\lm$ be the $\C T$-module that is equal
to $\C$ as a vector space and has the following  $T$-action: $tc=\lm(t)c$.
Then we have the line bundle $\C_\lm\times_T E_T\to B_T$ denoted by $\mathcal L_T(\lm)$,
where $E_T\to B_T$ is a universal principal $T$-bundle. We get the map ${\mathfrak X}(T)\to H^2_T(\pt)=H^2(B_T,\k)$ given by
$\lm\mapsto c_1(\mathcal L_T(\lm))$, where $c_1$ denotes the first Chern class,
which extends to the isomorphism with the symmetric algebra:
\begin{equation}\label{eq:pt:1}
\mathop{\mathrm{Sym}}({\mathfrak X}(T)\otimes_\Z\C)\ito S.
\end{equation}

Similarly, let $\mathfrak X(K)$ be the group of continuous homomorphisms $K\to\C^\times$.
For each $\lm\in\mathfrak X(K)$, we have the bundle $\mathcal L_K(\lm)$ similar to $\mathcal L_T(\lm)$.
Therefore, we have the isomorphism
\begin{equation}\label{eq:pt:1.5}
\mathop{\mathrm{Sym}}(\mathfrak X(K)\otimes_\Z\C)\ito S
\end{equation}
induced by $\lm\mapsto c_1(\mathcal L_K(\lm))$. In what follows, we identify $\mathfrak X(T)$ with $\mathfrak X(K)$
via the restriction. Then both isomorphisms~(\ref{eq:pt:1}) and~(\ref{eq:pt:1.5}) become equal.
Note that the Weyl group $W$ acts on $\mathfrak X(K)$ and $\mathfrak X(T)$ by $(w\lm)(t)=\lm(\dot w^{-1}t\dot w)$.

We are free to choose a universal principal $K$-bundle $E_K\to B_K$ to compute $S=H^\bullet(B_K,\k)$.
We assume that the $K$-action on $E_K$ can be extended to a continuous $C$-action.
The quotient map $E^\infty\to E^\infty/K$ is an example of such a bundle.
The map $\rho_w:E_K\to E_K$ defined by $\rho_w(e)=\dot we$ factors through the action
of $K$ and we get the map $\rho_w/K:E_K/K\to E_K/K$.
This map induces the ring homomorphism $(\rho_w/K)^*:S\to S$.
It is easy to check that the pullback of $\mathcal L_K(\lm)$ along $\rho_w/K$ is $\mathcal L(w^{-1}\lambda)$.
Therefore under identification~(\ref{eq:pt:1.5}), we get
\begin{equation}\label{eq:pt:2}
(\rho_w/K)^*(u)=w^{-1}u
\end{equation}
for any $u\in S$.

For a finite space $X$ with the discrete topology and trivial action of $K$,
we identify $H^\bullet_T(X,\k)$
with $S(X)$.
More exactly, let $x\in X$ be an arbitrary point. Consider the map $j_x:E_K/K\to X\times_KE_K$
given by $j_x(Ke)=K(x,e)$. Then any element $h\in H^n_K(X,\k)$ is identified with the function
$x\mapsto j_x^*h$. A similar identification is possible for a finite discrete space $X$ with the trivial action of $T$.

\subsection{Categories of Bott-Samelson varieties}\label{Categories of Bott-Samelson varieties}
We are going to recall the definition of the folding category $\widetilde\Seq$ from~\cite{CBSV}.
The objects of this category are sequences (on the initial intervals of natural numbers) of simple reflections.
Each morphism $(s_1,\ldots,s_n)\to(\tilde s_1,\ldots,\tilde s_{\tilde n})$ is a triple $(p,w,\phi)$ such that

\begin{enumerate}\itemsep=6pt
\item\label{fcat:1} $p:\{1,\ldots,n\}\to\{1,\ldots,\tilde n\}$ is a monotone embedding;
\item\label{fcat:2}  $w\in W$;
\item\label{fcat:3}  $\phi:\Gamma(s)\to\Gamma(\tilde s)$ is a map such that
      \begin{equation}\label{eq:bt9:1}
      \phi(\gamma)^{p(i)}\tilde s_{p(i)}(\phi(\gamma)^{p(i)})^{-1}=w\gamma^is_i(\gamma^i)^{-1}w^{-1},
      \quad
      \phi(\f_i\gamma)=\f_{p(i)}\phi(\gamma)
      \end{equation}
      for any $\gamma\in\Gamma(s)$ and $i=1,\ldots,n$.
\end{enumerate}
In the last property, we used the folding operators $\f_i$ introduced in Section~\ref{Galleries}.

Under some restrictions on the ring of coefficients, each morphism $(p,w,\phi)$
induces the map $\widetilde H((p,w,\phi))$ between $T$-equivariant cohomologies
of Bott-Samelson varieties, see~\cite[Section 3.5]{CBSV}.
The existence of the contravariant functor $\widetilde H$ followed in~\cite{CBSV} from
Theorem~1 proved in that paper with the help of M.\,H\"arterich's criterion for the image
of the localization~\cite{Haerterich}. We are going to show how this theorem and thus the existence of $\widetilde H$ naturally follow
from the results obtained in this paper.

We consider here the following special case of the constructions of Section~\ref{Nested_fibre_bundles}:
$F=R$ and for any $r\in R$, we have $r_1=r_2$ and $v_r=1$ or $v_r=s_{r_1}$.
In this case, it follows from Theorem~\ref{theorem:1} that the map
$$
p^R:\BS_c(s,v)\to\BS_c(s^R,v^R)=\BS_c(s^R)
$$
is a homeomorphism, as each space $\BS_c(s_{f^m},v_{f^m})$ 
consists of the only point $[s_{f^m_1}]$.

\begin{lemma}
Let $(p,w,\phi)$ be a morphism from 
$s=(s_1,\ldots,s_n)$ to $\tilde s=(\tilde s_1,\ldots,\tilde s_{\tilde n})$ in the category $\widetilde\Seq$.
Then there exists a continuous map $\psi:\BS_c(s)\to\BS_c(\tilde s)$
such that $\psi^K=\phi$ and $\psi(ka)=\dot wk\dot w^{-1}a$ for any $k\in K$ and $a\in\BS_c(s)$.
\end{lemma}
\begin{proof} We set
$$
I=\{1,2,\ldots,\tilde n\},\quad R=\{(i,i)\suchthat i\in I\setminus\im p\}.
$$
We get $I^R=\im p$.
We define the map $v:R\to W$ by $v_{(i,i)}=\phi(\gamma)_i$ where $\gamma$ is an arbitrary element of $\Gamma(s)$.
This definition makes no confusion, as by the second formula of~(\ref{eq:bt9:1}), the element
$\phi(\gamma)_i$ does not depend on the choice of $\gamma$ if $i\notin\im p$.

Let us compute the sequence $\tilde s^R$. Let $\bar\gamma\in\Gamma_s$ be such that $\phi(\bar\gamma)_i=1$ for any
$i\in\im p$. This gallery exists and is unique. 
Then we have $\phi(\bar\gamma)^i=v^i$ for any $i\in\im p$.
By the first equation of~(\ref{eq:bt9:1}), we get
$$
\tilde s^R_{p(i)}=
v^{p(i)}\tilde s_{p(i)}(v^{p(i)})^{-1}=
\phi(\bar\gamma)^{p(i)}\tilde s_{p(i)}(\phi(\bar\gamma)^{p(i)})^{-1}
=w\bar\gamma^is_i(\bar\gamma^i)^{-1}w^{-1}=ws^{(\bar\gamma)}_iw^{-1}.
$$
Thus $\tilde s^Rp=(s^{(\bar\gamma)})^w$.

We have the following homeomorphisms:
\begin{equation}\label{eq:3}
\begin{tikzcd}[column sep=3.5ex]
\BS_c(\tilde s,v)\arrow{r}{p^R}[swap]{\sim}&\BS_c(\tilde s^R)\arrow{r}[swap]{\sim}&\BS_c(\tilde s^Rp)=\BS_c((s^{(\bar\gamma)})^w)&\arrow{l}{\sim}[swap]{d_w}\BS_c(s^{(\bar\gamma)})&\arrow{l}{\sim}[swap]{D_{\bar\gamma}}\BS_c(s).
\end{tikzcd}
\!\!\!\!
\end{equation}
As $\BS_c(\tilde s,v)\subset\BS_c(\tilde s)$, we obtain the embedding $\psi:\BS_c(s)\to\BS_c(\tilde s)$,
reading the above diagram from right to left.
Clearly, $\psi(ka)=\dot wk\dot w^{-1}\psi(a)$ for any $a\in\BS_c(s)$ and $k\in K$, as the same property holds for $d_w$
and the other isomorphisms are $K$-equivariant.
So we can consider the restriction
$\psi^K:\Gamma(s)\to\Gamma(\tilde s)$.
Let us prove that $\psi^K=\phi$.
Let $\gamma\in\Gamma(s)$. Then by the definition of $v$, we get that $\phi(\gamma)$ belongs to $\BS_c(\tilde s,v)^K$
and is $F$-balanced (under the identifications of Section~\ref{Definition via the compact torus}).
Let us write how the galleries $\phi(\gamma)$ and $\gamma$ are mapped
towards each other in diagram~(\ref{eq:3}) as follows:
$$
\begin{tikzcd}
\phi(\gamma)\arrow[mapsto]{r}&\delta\arrow[mapsto]{r}&\lambda,\quad \mu&\arrow[mapsto]{l}\rho&\arrow[mapsto]{l}\gamma.
\end{tikzcd}
$$
Our aim is obviously to prove that $\lm=\mu$. To this end, take an arbitrary $i=1,\ldots,n$.
Then
$$
\lm_i=\delta_{p(i)}=\phi(\gamma)_{p(i)}^F=v^{p(i)}\phi(\gamma)_{p(i)}(v^{p(i)})^{-1}
=\phi(\bar\gamma)^{p(i)}\phi(\gamma)_{p(i)}\big(\phi(\bar\gamma)^{p(i)}\big)^{-1}.
$$
and
$$
\mu_i=w\rho_iw^{-1}=w\bar\gamma^{i-1}\gamma_i(\bar\gamma^i)^{-1}w^{-1}.
$$

If $\phi(\gamma)_{p(i)}=1$, then $\gamma_i=\bar\gamma_i$ and therefore $\mu_i=1=\lm_i$.
Suppose now that $\phi(\gamma)_{p(i)}=\tilde s_{p(i)}$. Then $\gamma_i=\bar\gamma_is_i$.
We get
$$
\lm_i=\phi(\bar\gamma)^{p(i)}\tilde s_{p(i)}\big(\phi(\bar\gamma)^{p(i)}\big)^{-1},\quad
\mu_i=w\bar\gamma^is_i(\bar\gamma^i)^{-1}w^{-1}.
$$
These elements are equal by the first
equation of~(\ref{eq:bt9:1}).
\end{proof}

Now let us apply this lemma to the computation of the equivariant cohomologies.
The map $\psi\times\rho_w:\BS_c(s)\times E^\infty\to\BS_c(\tilde s)\times E^\infty$ maps $K$-orbits to $K$-orbits.
Indeed, for any $k\in K$, $a\in\BS_c(s)$ and $e\in E^\infty$, we get
\begin{multline*}
(\psi\times\rho_w)(k(a,e))=(\psi\times\rho_w)((ka,ke))=(\psi(ka),\rho_w(ke))\\
=(\dot wk\dot w^{-1}\psi(a),\dot w ke)
=\dot wk\dot w^{-1}(\psi(a),\dot we)=\dot wk\dot w^{-1}(\psi\times\rho_w)((a,e)).
\end{multline*}
Hence we get the quotient map $\psi\times_K\rho_w:\BS_c(s)\times_K E^\infty\to\BS_c(\tilde s)\times_K E^\infty$.
Similarly, we get the map $\phi\times_K\rho_w:\Gamma(s)\times_K E^\infty\to\Gamma(\tilde s)\times_K E^\infty$.
We have the following commutative diagram:
\begin{equation}\label{eq:4}
\begin{tikzcd}
[column sep=15ex]
H^\bullet_K(\BS_c(s),\k)\arrow{d}&H^\bullet_K(\BS_c(\tilde s),\k)\arrow{d}\arrow{l}[swap]{(\psi\times_K\rho_w)^*}\\
H^\bullet_K(\Gamma(s),\k)&H^\bullet_K(\Gamma(\tilde s),\k)\arrow{l}[swap]{(\phi\times_K\rho_w)^*}
\end{tikzcd}\!\!\!
\end{equation}
where the vertical arrow are restrictions. In the rest of this section, we identify
$K$- and $T$-equivarinat cohomologies.

We can now give a different proof of Theorem~1 from~\cite{CBSV}.
To do it, we need to compute the map $(\phi\times_K\rho_w)^*$ in the bottom arrow:
\begin{multline*}
(\phi\times_K\rho_w)^*(g)(\gamma)=
j_\gamma^*(\phi\times_K\rho_w)^*(g)=
((\phi\times_K\rho_w)j_\gamma)^*(g)
=(j_{\phi(\gamma)}(\rho_w/K))^*(g)\\
=(\rho_w/K)^*j_{\phi(\gamma)}^*(g)=
(\rho_w/K)^*(g(\phi(\gamma)))
=w^{-1}g(\phi(\gamma)),
\end{multline*}
where we applied~(\ref{eq:pt:2}) to obtain the last equality.
Thus in the notation of~\cite[Theorem~1]{CBSV}, we get
$$
(\phi\times_K\rho_w)^*(g)=g_{(p,w,\phi)}.
$$
We have just given another prove of this theorem: if $g$ is in the image of the right vertical arrow
of~(\ref{eq:4}), then $g_{(p,w,\phi)}$ is in the image of
the left vertical arrow. Note that we did not impose in this prove any restrictions on
the commutative ring $\k$.

We can identify the top arrow of~(\ref{eq:4}) with the map $\widetilde H((p,w,\phi))$
from~\cite[Section 3.5]{CBSV} under some conditions on $\k$
and the root system implying the localization theorem.

The category $\widetilde\Seq_{\mathbf f}$ is defined similarly to the category $\widetilde\Seq$, see~\cite[Section 5.3]{CBSV}.
The main difference is that the objects are pairs $(s,x)$, where $s$ is a sequence of simple reflections
and $x\in W$. Our preceding arguments give the following topological proof of Theorem~5 from~\cite{CBSV}.
Let $(p,w,\phi):(s,x)\to(\tilde s,\tilde x)$
be a morphism of the category $\widetilde\Seq_{\mathbf f}$.
We assume that $\Gamma(s,x)\ne\emptyset$. Then by~\cite[Lemma~19]{CBSV}, there exists
a map $\bar\phi:\Gamma(s)\to\Gamma(\tilde s)$ such that $(p,w,\bar\phi)$ is a morphism
of $\widetilde\Seq$ and the restriction of $\bar\phi$ to $\Gamma(s,x)$ is $\phi$.
Choosing $\bar\gamma\in\Gamma(s)$ so that $\bar\phi(\bar\gamma)_i=1$ for any $i\in\im p$
and reading~(\ref{eq:3}) from right to left, we construct the map $\psi:\BS_c(s)\to\BS_c(\tilde s)$.
It is easy to check that it takes $\BS_c(s,x)$ to $\BS_c(\tilde s,\tilde x)$.
Indeed, looking at~(\ref{eq:3}), we conclude that we have to prove the equality
$\tilde x(\bar\phi(\bar\gamma)^{\max})^{-1}=wx(\bar\gamma^{\max})^{-1}w^{-1}$.
This can be done by induction, that is, we are going to prove that
\begin{equation}\label{eq:bt7:12}
\tilde x(\bar\phi(\gamma)^{\max})^{-1}=wx(\gamma^{\max})^{-1}w^{-1}
\end{equation}
for any $\gamma\in\Gamma(s)$. First, we take any $\gamma\in\Gamma(s,x)$. Then $\bar\phi(\gamma)=\phi(\gamma)\in\Gamma(\tilde s,\tilde x)$
and equality~(\ref{eq:bt7:12}) is reduced to $1=1$. Now suppose that~(\ref{eq:bt7:12}) is true for some $\gamma\in\Gamma(s)$.
Applying~(\ref{eq:bt9:1}), we get

\begin{multline*}
\tilde x(\bar\phi(\f_i\gamma)^{\max})^{-1}=\tilde x((\f_{p(i)}\bar\phi(\gamma))^{\max})^{-1}
=\tilde x(\bar\phi(\gamma)^{p(i)}\tilde s_{p(i)}(\bar\phi(\gamma)^{p(i)})^{-1}\bar\phi(\gamma)^{\max})^{-1}\\
=\tilde x(\bar\phi(\gamma)^{\max})^{-1}w\gamma^is_i(\gamma^i)^{-1}w^{-1}
=wx(\gamma^{\max})^{-1}
\gamma^is_i(\gamma^i)^{-1}w^{-1}\\
=wx(\gamma^is_i(\gamma^i)^{-1}\gamma^{\max})^{-1}
w^{-1}
=wx((\f_i\gamma)^{\max})^{-1}w^{-1}.
\end{multline*}

As we can reach any combinatorial gallery of $\Gamma(s)$ from any other combinatorial galley of this set
by successively applying the folding operators, we get~(\ref{eq:bt7:12}) for all $\gamma\in\Gamma(s)$.
Now we can prove Theorem~5 from~\cite{CBSV} exactly in the same way as we proved Theorem~1 above.

\section{Approximation by compact spaces}\label{Approximation by compact spaces}

\subsection{Vietoris-Begle theorem}

First, remember the following classical result.

\begin{proposition}[\mbox{\cite[Theorem IV.I.6]{Iversen}}]\label{proposition:coh:1}
Let $f:X\to Y$ be a fibre bundle whose fibre is homotopic to a compact Hausdorff space.
Then for each $y\in Y$, the canonical map
\begin{equation}\label{eq:coh:10}
(R^nf_*\csh{\k}{X})_y\to H^n(f^{-1}(y),\k)
\end{equation}
is an isomorphism for any $n$.
\end{proposition}

From this proposition, we get the following version of the Vietoris-Begle mapping theorem.

\begin{lemma}\label{lemma:coh:3}
Let $f:X\to Y$ be a fibre bundle whose fibre $F$ is connected and homotopic to a compact Hausdoff space.
Suppose that there exists an integer $N$ or $N=\infty$ such that $H^n(F,\k)=0$
for any $0<n<N$. Then the canonical map
\begin{equation}\label{eq:coh:11}
H^n(Y,\k)\to H^n(X,\k)
\end{equation}
is an isomorphism for any $n<N$.
\end{lemma}
\begin{proof} We generally follow the lines of the proof~\cite[Theorem III.6.4]{Iversen}.
First we note that the morphism $a:\csh{\k}{Y}\to f_*\csh{\k}{X}$
given by the adjunction unit is an isomorphism.
Indeed, for any open $V\subset Y$ and $s\in\csh{\k}{Y}(V)$, the map $a(V)(s)$ is the
composition $f^{-1}(V)\stackrel f\to V\stackrel s\to \k$. If we restrict $a$ to a point $y\in Y$
and compose it with~(\ref{eq:coh:10}) for $n=0$:
$$
\k\to(f_*\csh{\k}{X})_y=(R^0f_*\csh{\k}{X})_y\stackrel\sim\to H^0(f^{-1}(y),\k),
$$
then we get the map that takes $\gamma\in \k$ to the function on $f^{-1}(y)$ taking constantly the value $\gamma$.
As $f^{-1}(y)$ is connected, this map is an isomorphism. Hence the restriction of $a$ to $y$
is also an isomorphism.

Now let $\csh {\k}{X}\to I^\bullet$ be an injective resolution. Applying $f_*$, we get the sequence
$$
0\to f_*\csh {\k}{X}\to f_*I^0\to f_*I^1\to\cdots\to f_*I^N,
$$
which is exact by Proposition~\ref{proposition:coh:1} (see the proof of~\cite[Theorem III.6.4]{Iversen}).
This sequence can be completed to an injective resolution $f_*\csh {\k}{X}\to J^\bullet$,
where $J^n=f_*I^n$ for $n\le N$. Composing with $a$, we also get an injective resolution $\csh{\k}{Y}\to J^\bullet$.

Now we want to describe~(\ref{eq:coh:11}) applying the definition from~\cite[II.5]{Iversen}. The comparison theorem
for injective resolutions, yields (dashed arrows) a commutative diagram with exact rows
$$
\begin{tikzcd}[column sep=small]
0\arrow{r}&f^*\csh{\k}{Y}\arrow[equal]{d}\arrow{r}&f^*f_*I^0\arrow{r}\arrow{d}&f^*f_*I^1\arrow{r}\arrow{d}&\cdots\arrow{r}&f^*f_*I^N\arrow{r}\arrow{d}&f^*J^{N+1}\arrow{r}\arrow[dashed]{d}&f^*J^{N+2}\arrow{r}\arrow[dashed]{d}&\cdots\\
0\arrow{r}&\csh{\k}{X}\arrow{r}&I^0\arrow{r}&I^1\arrow{r}&\cdots\arrow{r}&I^N\arrow{r}&I^{N+1}\arrow{r}&I^{N+2}\arrow{r}&\cdots
\end{tikzcd}
$$
where the solid (not dashed) vertical arrows come from the counit of the adjunction.
The zigzag identity shows that the map
$$
\Gamma(Y,J^n)\to\Gamma(X,f^*J^n)\to\Gamma(X,I^n)=\Gamma(Y,f_*I^n)
$$
is the identity map for $n\le N$. Hence~(\ref{eq:coh:11}) is an isomorphism for $n<N$.
\end{proof}

\subsection{Approximation}
The following result is well known.
\begin{proposition}\label{proposition:3}
\begin{enumerate}\itemsep=8pt
\item\label{proposition:3:p:1} $E^N$ is simply-connected if $N>\r$.
\item\label{proposition:3:p:2} $H^n(E^N,\k)=0$ for $0<n<2(N-\r)+1$.
\item\label{proposition:3:p:3} $H^n(E^N,\k)$ is free for any $n$.
\end{enumerate}
\end{proposition}

To compute the modules $H^n_K(X,\k)$, one can approximate $E^\infty$ by $E^N$ with $N$ big enough
as follows.

\begin{lemma}\label{lemma:15}
Let $X$ be a $K$-space. Then $H^n_K(X,\k)\cong H^n(X\times_KE^N,\k)$ for $n<2(N-\r)+1$.
These isomorphisms are natural along
continuous $K$-equivariant maps
and with respect to with the cup product.
\end{lemma}
\begin{proof} For $n<2(N-\r)+1$, have the diagram similar to~(\ref{eq:bt6:2})
$$
\begin{tikzcd}
{}&H^n((X\times E^\infty\times E^N)/K,\k)&\\
\arrow{ru}{(p_{12}/K)^*}[swap]{\sim}H^n_K(X,\k)&&H^n(X\times_K E^N)\arrow{lu}[swap]{(p_{13}/K)^*}{\sim}
\end{tikzcd}
$$
Both arrows are isomorphisms by Proposition~\ref{proposition:3}\ref{proposition:3:p:2} and
Lemma~\ref{lemma:coh:3}. The two remaining assertions can be checked routinely.
\end{proof}

\begin{lemma}\label{lemma:bt6:4}
Let $X_1,\ldots,X_m$ be $K$-spaces. Then
$$
H^n\(\prod_{i=1}^mX_i\times_KE^\infty,\k\)\cong H^n\(\prod_{i=1}^mX_i\times_KE^N,\k\)
$$
for $n<2(N-\r)+1$. These isomorphisms are natural with respect to the cup product and the projections to factors.
In particular,
we get the following commutative diagram:
$$
\begin{tikzcd}
\displaystyle
\mathop{{\bigotimes{}_\k}}\limits_{n_1+\cdots+n_m=n} H^{n_i}\(X_i\times_KE^\infty,\k\)\arrow{r}\arrow[equal]{d}[swap]{\wr}&\displaystyle H^n\(\prod_{i=1}^mX_i\times_KE^\infty,\k\)\arrow[equal]{d}{\wr}\\
\displaystyle
\mathop{{\bigotimes{}_\k}}\limits_{n_1+\cdots+n_m=n} H^{n_i}\(X_i\times_KE^N,\k\)\arrow{r}&\displaystyle H^n\(\prod_{i=1}^mX_i\times_KE^N,\k\)
\end{tikzcd}
$$
\end{lemma}
\noindent
where the horizontal arrow are given by the cross product.
\begin{proof}
The result follows if we consider the following direct products of fibre bundles:
$$
\begin{tikzcd}
{}&\displaystyle\prod_{i=1}^m(X_i\times E^\infty\times E^N)/K\arrow{dr}\arrow{dl}&\\
\displaystyle\prod_{i=1}^mX_i\times_KE^\infty&&\displaystyle\prod_{i=1}^mX_i\times_KE^N
\end{tikzcd}
$$
The left bundle has fibre $(E^N)^m$ and the right one fibre $(E^\infty)^m$. The last space
is contractible and $H^n((E^N)^m,\k)=0$ for $0<n<2(N-\r)+1$ by the K\"unneth formula.
It remains to take cohomologies and apply Lemma~\ref{lemma:coh:3}.
\end{proof}

We want to prove the results similar to Proposition~\ref{proposition:3} for the quotient $E^N/K$.
\begin{lemma}\label{lemma:12}
\begin{enumerate}
\item\label{lemma:12:p:1} $E^N/K$ is simply-connected if $N>\r$.\\[-6pt]
\item\label{lemma:12:p:2} Suppose that $N>\r$. Then for $n<2(N-\r)+1$, the $\k$-module $H^n(E^N/K,\k)$ is free of finite rank
                          and equals zero if $n$ is odd.
\end{enumerate}
\end{lemma}
\begin{proof}\ref{lemma:12:p:1} As $K$ is connected, Part~\ref{lemma:12:p:1} follows from Proposition~\ref{proposition:3}\ref{proposition:3:p:1},
and the long exact sequence of homotopy groups:
$$
\{1\}=\pi_1(E^N)\to \pi_1(E^N/K)\to \pi_0(K)=\{1\}.
$$

\ref{lemma:12:p:2} By Lemma~\ref{lemma:15}, we get
$
H^n(E^N/K,\k)\cong H_K^n(\pt,\k)=S.
$
As the the latter module vanishes in odd degrees, the result follows.
\end{proof}

\begin{lemma}\label{lemma:14}
Let $X$ be a $K$-space having an affine paving. Then for $N<\infty$ and any $n\le2(N-\r)-1$, the $\k$-module $H_c^n(X\times_K E^N,\k)$
is free of finite rank and is zero if $n$ is odd.
\end{lemma}
\begin{proof} It suffices to consider only the case $n\ge0$. Then we have $N>\mathfrak r$.
Let us consider the Leray spectral sequence with compact support for the canonical projection $X\times_KE^N\to E^N/K$
As $E^N/K$ is simply connected, it has the following second page:
$$
E_2^{p,q}=H^p(E^N/K,H_c^q(X,\k)).
$$
By Lemma~\ref{lemma:12}\ref{lemma:12:p:2}, we get that $E_2^{p,q}=0$
except the following cases: $p\ge 2(N-\r)+1$; both $p$ and $q$ are even.
Moreover, $E_2^{p,q}$ is free of finite rank for $p<2(N-\r)+1$.

The differentials coming to and starting from $E_a^{p,q}$ are thus zero if $a\ge2$ and $p+q+1\le2(N-\r)$.
So we get $E_\infty^{p,q}=E_2^{p,q}$ for $p+q+1\le2(N-\r)$.
It follows from the spectral sequence that $H_c^n(X\times_KE^N,\k)=0$ for odd $n\le2(N-\r)-1$.
For even $n\le2(N-\r)-1$, the module $H_c^n(X\times_KE^N,\k)$ is filtered by the free of finite rank modules $E_2^{p,q}$
with even nonnegative $p$ and $q$ such that $p+q=n$. As $\mathop{\mathrm{Ext}}\nolimits^1(\k,\k)=0$,
the module $H_c^n(X\times_KE^N,\k)$ is also free of finite rank.
\end{proof}

\section{Cohomology of $\BS_c(s)$}\label{Cohomology of BS_c(s)} In this section, we are going to describe a
set of multiplicative generators of $H_K^\bullet(\BS_c(s),\k)$.
%
%
%

\subsection{Twisted actions of $S$} Suppose that $E_K\to B_K$ is a universal principal $K$-bundle
such that the $K$-action on $E_K$ can be extended to a continuous $C$-action.
An example of such a bundle is the quotient map $E^\infty\to E^\infty/K$ (see, Section~\ref{Stiefel manifolds}).
Let $s:I\to\mathcal T(W)$ be a sequence
and $i$ be an element of $I\cup\{-\infty\}$. We define the map $\Sigma(s,i,w):\BS_c(s)\times_KE_K\to E_K/K$ by
\begin{equation}\label{eq:bt9:2}
K([c],e)\mapsto K(c^i\dot w)^{-1}e.
\end{equation}
The reader can easily check that this map is well-defined and continuous. Taking cohomologies, we get the map
$$
\Sigma(s,i,w)^*:S\to H_K^\bullet(\BS_c(s),\k).
$$
This map induces the action of $S$ on $H^\bullet(\BS_c(s),\k)$ by
$u\cdot h=\Sigma(s,i,w)^*(u)\cup h$. For $w=1$ and $i=-\infty$, we get the canonical
action of $S$.
%
Note that these actions are independent of the choice of the universal principal $K$-bundle
and of the action of $C$.

We also consider the finite dimensional version of these maps. Let $\Sigma^N(s,i,w):\BS_c(s)\times_KE^N_K\to E^N_K/K$
be the map given by~(\ref{eq:bt9:2}). Here $N$ may be an integer greater than or equal to $\r$ or $\infty$.
Note that $\Sigma^\infty(s,i,w)$ is a representative of $\Sigma(s,i,w)$.

\subsection{Embeddings of Borel constructions}\label{Embeddings of Borel constructions} We are going to prove the following result,
which use for induction. Remember that we denote truncated sequences by the prime.

Let $\r\le N$, $s:I\to\mathcal T(W)$ be a nonempty sequence,
$w\in W$ and $\tr:\BS_c(s)\to\BS_c(s')$ be the truncation map $\tr([c])=[c']$.
We consider the map
$$
\chi_w^N=(\tr\times_K\id)\boxtimes\Sigma^N(s,\max I,w)
$$
from $\BS_c(s)\times_K E^N$ to $L^N=(\BS_c(s')\times_K E^N)\times(E^N/K)$.
We also abbreviate $L=L^\infty$ and $\chi_w=\chi_w^\infty$.

\begin{lemma}\label{lemma:bt6:1} For $N<\infty$,
the map $\chi_w^N$ is a topological embedding.
Its image consists of the pairs $(K([d],e),K\tilde e)$ such that
\begin{equation}\label{eq:bt6:13}
C_{w^{-1}s_{\max I}w}\tilde e=C_{w^{-1}s_{\max I}w}\Sigma^N(s',\max I',w)(K([d],e)).
\end{equation}
\end{lemma}
\begin{proof} We denote $i=\max I$ for brevity. To prove the claim about the topological embedding,
it suffices to prove that the above map is injective.
Suppose that two orbits $K([c],e)$ and $K([\tilde c],\tilde e)$ are mapped to the same pair.
As $K([c'],e)=K([\tilde c'],\tilde e)$, we can assume that  $[c']=[\tilde c']$ and $e=\tilde e$.
Therefore, without generality we can assume that $c'=\tilde c'$.

We get $\Sigma^N(s,i,w)(K([c],e))=\Sigma^N(s,i,w)(K([\tilde c],e))$. Let us write this equality as follows:
$$
\dot w^{-1}c_i^{-1}(c^{i-1})^{-1}e=k\dot w^{-1}\tilde c_i^{-1}(c^{i-1})^{-1}e
$$
for some $k\in K$. As $C$ acts freely on $E^N$, we get $\dot w^{-1}c_i^{-1}=k\dot w^{-1}\tilde c_i^{-1}$,
whence $\tilde c_i=c_i\dot wk\dot w^{-1}$. As $\dot wk\dot w^{-1}\in K$, we get $[\tilde c]=[c]$.

Let us check that any element of the image of
$\chi_w^N$
satisfies~(\ref{eq:bt6:13}).
Let $K([c],e)$ be an element of $\BS_c(s)\times_K E^N$. It is mapped to $(K([c'],e),K(c^i\dot w)^{-1}e)$. We get
$$
(c^i\dot w)^{-1}e=\dot w^{-1}c_i^{-1}(c^{i-1})^{-1}e=(\dot w^{-1}c_i^{-1}\dot w)(c^{i-1}\dot w)^{-1}e
=\dot w^{-1}c_i^{-1}\dot w\,\Sigma^N(s',i-1,w)(K([c'],e)).
$$
It remains to note that $\dot w^{-1}c_i^{-1}\dot w\in\dot w^{-1}C_{s_i}\dot w=C_{w^{-1}s_iw}$.

Conversely, suppose that a pair $(K([d],e),K\tilde e)$ of $L^N$ satisfies~(\ref{eq:bt6:13}).
Then there exists an element $c\in\dot w^{-1}C_{s_i}\dot w$ such that
$c\tilde e=(d^{i-1}\dot w)^{-1}e$. We define the sequence $c:I\to\mathcal T(W)$ by
$$
c_j=
\left\{
\begin{array}{ll}
d_j&\text{ if }j<i;\\[6pt]
\dot wc\dot w^{-1}&\text{ otherwise}.
\end{array}
\right.
$$
We get
$$
(\tr\times_K\id)(K([c],e))=K([c'],e)=K([d],e).
$$
On the other hand
$$
\Sigma^N(s,i,w)(K([c],e))=K(c^i\dot w)^{-1}e=K\dot w^{-1}c_i^{-1}(c^{i-1})^{-1}e
=Kc^{-1}(d^{i-1}\dot w)^{-1}e=K\tilde e.
$$
\end{proof}

\subsection{The difference $L^N\setminus\im\chi_w^N$}
We study the cohomology of this difference by considering it as the total space
of a fibre bundle. Here and in what follows $\chi_w^N$ and $L^N$ are as in Lemma~\ref{lemma:bt6:1}.

\begin{lemma}\label{lemma:bt6:2}
Let $\r\le N<\infty$.
The projection to the first component $\omega:L^N\setminus\im\chi_w^N\to BS_c(s')\times_K E^N$
is a fibre bundle.
\end{lemma}
\begin{proof} We denote $i=\max I$ and $\Sigma=\Sigma^N(s',i-1,w)$ for brevity.
Let $b$ be an arbitrary element of $\BS_c(s')\times_K E^N$.
The right action of the unitary group $U(N)$ on $E^N$ induces the right action of $U(n)$ on $E^N/K$.
We denote this action by $\cdot$.
Let $t:U(N)\to E^N/K$ be the map $t(g)=\Sigma(b)\cdot g$.
As $U(n)$ acts transitively on $E^N$, it acts transitively on $E^N/K$ and $t$ is a fibre bundle.
Therefore, there exists an open neighbourhood of $V$ of $\Sigma(b)$ in $E^N/K$
and a continuous section $g:V\to U(N)$ of $t$.
Hence for any $u\in V$, we get
\begin{equation}\label{eq:bt6:14}
u=t(g(u))=\Sigma(b)\cdot g(u).
\end{equation}
We define $H=\Sigma^{-1}(V)$.
It is an open subset of $BS_c(s')\times_K E^N$ containing~$b$.
We construct the map $\phi:H\times\omega^{-1}(b)\to L^N$ by
$$
\begin{tikzcd}
\big(h,(b,Ke)\big)\arrow[mapsto]{r}{\phi}&\Big(h,Ke\cdot g(\Sigma(h))\Big).
\end{tikzcd}
$$
Suppose that the right-hand side of the above formula belongs to $\im\chi_w^N$.
By Lemma~\ref{lemma:bt6:1}, we get
$$
C_{w^{-1}s_iw}e\cdot g(\Sigma(h))=C_{w^{-1}s_iw}\Sigma(h).
$$
Thus
$$
C_{w^{-1}s_iw}e=C_{w^{-1}s_iw}\Sigma(h)\cdot g(\Sigma(h))^{-1}
$$
By~(\ref{eq:bt6:14}) with $u=\Sigma(h)$, we get
$$
C_{w^{-1}s_iw}e=C_{w^{-1}s_iw}\Sigma(b).
$$
By Lemma~\ref{lemma:bt6:1}, this contradicts the fact that $(b,Ke)\notin\im\chi_w^N$. Thus we actually have the map
$\phi:H\times\omega^{-1}(b)\to\omega^{-1}(H)$.
It is easy to see that $\phi$ is a homeomorphism. Indeed the inverse map $\omega^{-1}(H)\to H\times\omega^{-1}(b)$
is given by
$$
(h,K\tilde e)\mapsto\Big(h,\big(b,K\tilde e\cdot g(\Sigma(h))^{-1}\big)\Big).
$$
We get the following commutative diagram:
$$
\begin{tikzcd}
H\times\omega^{-1}(b)\arrow{rr}{\phi}[swap]{\sim}\arrow{rd}[swap]{p_1}&&\omega^{-1}(H)\arrow{ld}{\omega}\\
&H&
\end{tikzcd}
$$
Finally note that $\omega^{-1}(b)$ are homeomorphic for different $b$, as the space
$\BS_c(s')\times_K E^N$ is connected and compact.
\end{proof}

\subsection{Compliment to a fibre}
We are going now to study the cohomology of the difference $L^N\setminus\im\chi_w^N$.
To this end, we need the following general result.

\begin{lemma}\label{lemma:even_bundles}
Let $X$ be a locally compact Hausdorff space, $\pi:X\to Y$ be a fibre bundle with fibre $Z$,
$y\in Y$ be a point and $\k$ be a commutative ring.
Suppose that $Y$ is compact, Hausdorff, connected and simply connected,
all $H_c^n(Z,\k)$ are free $\k$-modules of finite rank,
$H_c^n(Z,\k)=0$ for odd $n$ and $H_c^n(Y,\k)=0$ for odd $n\le N$.
\medskip
\begin{enumerate}
\item\label{lemma:even_bundles:p:1} The restriction map
$H_c^n(X,\k)\to H_c^n(\pi^{-1}(y),\k)$ is surjective for all
$n<N$.\\[-3pt]
\item\label{lemma:even_bundles:p:2} $H_c^n(X\setminus\pi^{-1}(y),\k)=0$ for odd
$n<N$.\\[-3pt]
\item\label{lemma:even_bundles:p:3} If $H_c^n(Y,\k)$ are free of finite rank for $n\le N$, then
the modules $H_c^n(X\setminus\pi^{-1}(y),\k)$ are also free of finite rank for
$n<N$.
\end{enumerate}
\end{lemma}
\begin{proof}\ref{lemma:even_bundles:p:1} Consider the following Cartesian diagram:
$$
\begin{tikzcd}
\pi^{-1}(y)\arrow[hook]{r}{\tilde\imath}\arrow{d}[swap]{\tilde\pi}&X\arrow{d}{\pi}\\
\{y\}\arrow[hook]{r}{i}&Y
\end{tikzcd}
$$
For the map $\pi$, we consider two Leray spectral sequences with compact support,
one for the complex $\csh{\k}{X}$ and the other one
for the complex $\tilde\imath_*\tilde\imath^*\csh{\k}{X}=\tilde\imath_!\csh{\k}{\pi^{-1}(y)}$.
Their second pages are
$$
{\arraycolsep=2pt
\begin{array}{rcl}
\displaystyle E_2^{p,q}&=&H_c^p(Y,\R^q\pi_!\csh{\k}{X}),\\[10pt]
\displaystyle \widetilde E_2^{p,q}&=&H_c^p\Big(Y,\R^q\pi_!\tilde\imath_!\csh{\k}{\pi^{-1}(y)}\Big)
\end{array}}
$$
respectively. As $\pi$ is a fibre bundle, $\R^q\pi_!\csh{\k}{X}$ is a locally constant sheaf with stalk $H_c^q(Z,\k)$.
%
%
%
As $Y$ is simply connected, we get $\R^q\pi_!\csh{\k}{X}=\csh{H_c^q(Z,\k)}{Y}=\csh{\k}{Y}^{\oplus m_q}$, where $m_q$
is the rank of the $\k$-module $H_c^q(Z,\k)$.
Hence $E_2^{p,q}=H_c^p(Y,\k)^{\oplus m_q}$. This module is zero except the following cases:
both $p$ and $q$ are even; $p>N$.
As the differentials coming to and starting from $E_a^{p,q}$ are zero if $a\ge2$ and $p+q+1\le N$,
we get $E_\infty^{p,q}=E_2^{p,q}$ for $p+q+1\le N$.

Now let us compute $\widetilde E_2^{p,q}$. As $i_!$ is exact, we get
$$
\R^q\pi_!\Big(\tilde\imath_!\csh{\k}{\pi^{-1}(y)}\Big)=\R^q(\pi\tilde\imath)_!\csh{\k}{\pi^{-1}(y)}
=\R^q(i\tilde\pi)_!\csh{\k}{\pi^{-1}(y)}=i_!\R^q\tilde\pi_!\csh{\k}{\pi^{-1}(y)}.
$$
The last sheaf is isomorphic to $i_!\csh{H_c^q(\pi^{-1}(y),\k)}{\{y\}}\simeq i_!\csh{H_c^q(Z,\k)}{\{y\}}\simeq i_!\csh{\k}{\{y\}}^{\oplus m_q}$.
Therefore,
$$
\widetilde E_2^{p,q}\simeq H_c^p\(Y,i_!\csh{\k}{\{y\}}^{\oplus m_q}{\{y\}}\)=H_c^p(\{y\},\k)^{\oplus m_q}=H^p(\{y\},\k)^{\oplus m_q}.
$$
Hence $\widetilde E_2^{p,q}=0$ unless $p=0$ and $q$ is even.
So the Leray spectral sequence with compact support
for $\tilde\imath_!\csh{\k}{\pi^{-1}(y)}$ collapses at the second page,
whence $\widetilde E_\infty^{p,q}=\widetilde E_2^{p,q}$.

We would like to use the functoriality of the Leray spectral sequence with compact support for $\pi$ along the map
$\tilde\eta_{\csh{\k}{X}}:\csh{\k}{X}\to\tilde\imath_*\tilde\imath^*\csh{\k}{X}=\tilde\imath_!\csh{\k}{\pi^{-1}(y)}$,
where $\tilde\eta:\id\to\tilde\imath_*\tilde\imath^*$ is the unit of adjunction.
Note that this map induces the restriction $H_c^n(X,\k)\to H_c^n(\pi^{-1}(y),\k)$.

Let us compute the induced map between the second pages. Applying $R\pi_!$, we get the morphism
$$
R\pi_!\tilde\eta_{\csh{\k}{X}}:
R\pi_!\csh{\k}{X}\to R\pi_!\tilde\imath_!\csh{\k}{\pi^{-1}(y)}.
$$
Taking the $q$th cohomology, we get the morphism of sheaves
$$
\R^q\pi_!\tilde\eta_{\csh{\k}{X}}:\R^q\pi_!\csh{\k}{X}\to\R^q\pi_!\tilde\imath_!\csh{\k}{\pi^{-1}(y)}.
$$
Now taking the $p$th cohomologies with compact support on $Y$, we get the map
$$
H_c^p\Big(Y,\R^q\pi_!\tilde\eta_{\csh{\k}{X}}\Big):E_2^{p,q}\to\widetilde E_2^{p,q}.
$$
This map is zero for $p\ne0$. We claim that it is an isomorphism for $p=0$. To prove it, let us write the morphism
$\R\pi_!\tilde\eta_{\csh{\k}{X}}$, applying the proper base change, as follows:
$$
R\pi_!\tilde\eta_{\csh{\k}{X}}:
R\pi_!\csh{\k}{X}\to R\pi_!\tilde\imath_!\csh{\k}{\pi^{-1}(y)}=i_! R\tilde\pi_!\tilde\imath^*\csh{\k}{X}=i_*i^*R\pi_!\csh{\k}{X}.
$$
It follows from the construction of the base change isomorphism (for example,~\cite[the Proof of Proposition~2.5.11]{KS}) that
this morphism is exactly $\eta_{R\pi_!\csh{\k}{X}}$, where $\eta:\id\to i_*i^*$
is the unit of adjunction.
Hence the morphism of sheaves
$$
\R^q\pi_!\tilde\eta_{\csh{\k}{X}}:\R^q\pi_!\csh{\k}{X}\to
i_*i^*\R^q\pi_!\csh{\k}{X}
$$
is conjugate to the identity morphism $i^*\R^q\pi_!\csh{\k}{X}\to i^*\R^q\pi_!\csh{\k}{X}$.
%
Hence the resulting map
$$
H_c^p\Big(Y,\R^q\pi_!\tilde\eta_{\csh{\k}{X}}\Big):E_2^{p,q}=H^p(Y,\k)^{\oplus m_q}\to H^p(\{y\},\k)^{\oplus m_q}=\widetilde E_2^{p,q}.
$$
is the restriction map. It is clearly 
an isomorphism for $p=0$, as $Y$ is connected.

Let us fix an integer $n\le N-1$. We have some decreasing filtrations
$$
\begin{array}{rcl}
H_c^n(X,\k)&=&F^0H_c^n(X,\k)\supset F^1H_c^n(X,\k)\supset\cdots\supset F^{n+1}H_c^n(X,\k)=0,\\[6pt]
H_c^n(\pi^{-1}(y),\k)&=&F^0H_c^n(\pi^{-1}(y),\k)\supset F^1H_c^n(\pi^{-1}(y),\k)\supset\cdots\supset F^{n+1}H_c^n(\pi^{-1}(y),\k)=0
\end{array}
$$
coming from the corresponding Leray spectral sequences. Here
$$
F^iH_c^n(X,\k)/F^{i+1}H_c^n(X,\k)\cong E_\infty^{i,n-i}=E_2^{i,n-i},
$$
$$
F^iH_c^n(\pi^{-1}(y),\k)/F^{i+1}H_c^n(\pi^{-1}(y),\k)\cong\widetilde E_\infty^{i,n-i}=\widetilde E_2^{i,n-i}.
$$
We get $F^iH_c^n(\pi^{-1}(y),\k)=0$ for $i>0$. The restriction map $H_c^n(X,\k)\to H_c^n(\pi^{-1}(y),\k)$
respects filtrations and is surjective as the map $E_2^{0,n}\to\widetilde E_2^{0,n}$
defined above is so. The kernel of the restriction map is $H_c^n(X,\k)\to H_c^n(\pi^{-1}(y),\k)$
is $F^1H_c^n(X,\k)$.

\ref{lemma:even_bundles:p:2} This claim follows from the following exact sequence for odd $n\le N-1$:
$$
H_c^{n-1}(X,\k)\twoheadrightarrow H_c^{n-1}(\pi^{-1}(y),\k)\to H_c^n(X\setminus\pi^{-1}(y),\k)\to H_c^n(X,\k)=0.
$$
Here the last equality follows from the first of the above spectral sequences.

\ref{lemma:even_bundles:p:3} Let $n$ be even and $n\le N-1$. Note that under the assumption of this
case all $E_2^{i,n-i}$ are free of finite rank for any $i=0,\ldots,n$. We have an exact sequence
$$
0=H_c^{n-1}(\pi^{-1}(y),\k)\to H_c^n(X\setminus\pi^{-1}(y),\k)\to H_c^n(X,\k)\to H_c^n(\pi^{-1}(y),\k)
$$
Hence $H_c^n(X\setminus\pi^{-1}(y),\k)\cong F^1H_c^n(X,\k)$ as the kernel of the restriction
$H_c^n(X,\k)\to H_c^n(\pi^{-1}(y),\k)$. This module is free of finite rank, as it has a finite
filtration with such quotients.
\end{proof}

\subsection{Generators} First, we consider restrictions of cohomologies.

\begin{lemma}\label{lemma:bt6:3}
Suppose that $2$ is invertible in $\k$.
\begin{enumerate}
\item\label{lemma:bt6:3:p:1} If $\r<N<\infty$, then $H_c^n(L^N\setminus\im\chi_w^N,\k)=0$ for odd $n<2(N-\r)-1$.\\[-6pt]
\item\label{lemma:bt6:3:p:2} If $\r<N<\infty$, then the restriction map $H^n(L^N,\k)\to H^n(\BS_c(s)\times_K E^N,\k)$
                          induced by $\chi_w^N$ is surjective for $n<2(N-\r)-2$.\\[-6pt]
\item\label{lemma:bt6:3:p:3} For any $n$, the restriction map $H^n(L,\k)\to H^n_K(\BS_c(s),\k)$ induced by $\chi_w$ is surjective.
\end{enumerate}
\end{lemma}
\begin{proof}\ref{lemma:bt6:3:p:1} We denote $i=\max I$ for brevity.
We write the Leray spectral sequence for cohomologies with compact support for
the fibre bundle $\omega$ as in Lemma~\ref{lemma:bt6:2}.
It has the following second page:
$$
E_2^{p,q}=H_c^p\big(\BS_c(s')\times_K E^N,H_c^q(\omega^{-1}(b),\k)\big),
$$
where $b$ is an arbitrary point of the space $\BS_c(s')\times_K E^N$,
which is simply connected. To prove the last statement, first note that $\BS_c(s')$
is simply connected as a Bott tower and then consider the long exact sequence of homotopy
groups for the fibre bundle $\BS_c(s')\times_K E^N\to E^N/K$ with fibre $\BS_c(s')$.

By Lemma~\ref{lemma:bt6:1}, the space $\omega^{-1}(b)$ is homeomorphic to the space $(E^N/K)\setminus\tau^{-1}(y)$, where
$\tau:E^N/K\to E^N/C_{w^{-1}s_iw}$ is the natural quotient map and $y$ is an arbitrary point of $E/C_{w^{-1}s_iw}$.
We know that $\tau$ is a fibre bundle with fibre $K\backslash C_{w^{-1}s_iw}=\{Kc\suchthat c\in C_{w^{-1}s_iw}\}\cong\C P^1$.
Let us apply the Gysin sequence to this fibre bundle as in~\cite[Example 5.C]{McCleary}.
As
$\C P^1$ is homeomorphic to the $2$-sphere,
we have an exact sequence
$$
\begin{tikzcd}
H^{n-3}(E^N/C_{w^{-1}s_iw},\k)\arrow{r}{z \cup}&H^n(E^N/C_{w^{-1}s_iw},\k)\arrow{r}{\tau^*}& H^n(E^N/K,\k)
\end{tikzcd}
$$
where $z$ is some element of $H^3(E^N/C_{w^{-1}s_iw},\k)$ such that $2z=0$. As $2$ is invertible in $\k$, we get $z=0$.
If $n$ is less than $2(N-\r)+1$ and is odd, then $H^n(E^N/K,\k)=0$ by Lemma~\ref{lemma:12}\ref{lemma:12:p:2}.
Hence and from the above exact sequence, we get $H^n(E^N/C_{w^{-1}s_iw},\k)=0$ for such $n$.
Note that $E^N/C_{w^{-1}s_iw}$ is simply connected and connected by an argument similar to Lemma~\ref{lemma:12}\ref{lemma:12:p:1}.
Applying Lemma~\ref{lemma:even_bundles} to $\tau$, we get that $H_c^q(\omega^{-1}(b),\k)=0$ for odd $q<2(N-\r)$.

By Lemma~\ref{lemma:14},
we get that $E_2^{p,q}$ is zero except the following cases: $p\ge2(N-\r)$; $q\ge2(N-\r)$;
$p$ and $q$ are both even.
It is easy to note that the differentials coming to and starting from $E_a^{p,q}$ are zero if $a\ge2$ and $p+q<2(N-\r)-1$.
Hence $E_\infty^{p,q}=E_2^{p,q}$ for $p+q<2(N-\r)-1$ and the claim immediately follows.

\ref{lemma:bt6:3:p:2} Let $n<2(N-\r)-2$. If $n$ is even, then  by the first part of this lemma,
we get an exact sequence
$$
H^n(L^N,\k)\to H^n(\BS_c(s)\times_KE^N,\k)\to H_c^{n+1}(L^N\setminus\im\chi_w^N,\k)=0.
$$
If $n$ is odd, then the restriction under consideration is surjective as $H^n(\BS_c(s)\times_KE^N,\k)=0$
by Lemma~\ref{lemma:14}.

\ref{lemma:bt6:3:p:3} This result follows from the previous part and Lemma~\ref{lemma:bt6:4} and the following commutative
diagram
$$
\begin{tikzcd}
H^n(L^N,\k)\arrow{r}{(\chi_w^N)^*}\arrow[equal]{d}[swap]{\wr}&H^n(\BS_c(s)\times_KE^N,\k)\arrow[equal]{d}{\wr}\\
H^n(L,\k)\arrow{r}{\chi_w^*}&H_K^n(\BS_c(s),\k)
\end{tikzcd}
$$
which holds for $n<2(N-\r)+1$.
\end{proof}

\begin{lemma}\label{lemma:bt6:5} Let $\k$ be a commutative ring of finite global dimension with invertible $2$.
Let $I$ be a finite totally ordered set and $s:I\to\mathcal T(W)$ and $w:I\cup\{-\infty\}\to W$ be arbitrary maps.
Then all elements $\Sigma(s,i,w_i)^*(h)$, where $i\in I\cup\{-\infty\}$
and $h\in S$, generate $H_K^\bullet(\BS_c(s),\k)$ as a ring with respect to addition and the cup product.
\end{lemma}
\begin{proof}
Let us apply the induction on the length of $s$. If $I=\emptyset$, then $\BS_c(s)$ is a singleton and
$\Sigma(s,-\infty,w_{-\infty})^*(h)=w_{-\infty}^{-1}h$ by~(\ref{eq:pt:2}). Obviously any element of $H_K^\bullet(\BS_c(s),\k)=S$
has this form.

Now suppose that $I\ne\emptyset$. We denote $i=\max I$ for brevity.
By the K\"unneth formula and Lemma~\ref{lemma:bt6:4}, we get
$$
H^\bullet(L,\k)\cong H^\bullet_K(BS_c(s'),\k)\otimes_\k S.
$$
By the inductive hypothesis, $H^\bullet(L,\k)$ is generated as a ring by elements
$p_1^*\Sigma(s',j,w_j)^*(h)$ and $p_2^*(h)$,
where $j\in I'\cup\{-\infty\}$, $h\in S$ and $p_1$ and $p_2$ are projections of
$L$ to its first and second component respectively.

Hence by Lemma~\ref{lemma:bt6:3}\ref{lemma:bt6:3:p:3} applied to $\chi_{w_i}$, we get that
$H^\bullet_K(\BS_c(s),\k)$ is generated as a ring by elements
$$
\chi_{w_i}^*p_1^*\Sigma(s',j,w_j)^*(h)=(\Sigma(s',j,w_j)p_1\chi_{w_i})^*(h)=\Sigma(s,j,w_j)^*(h),
$$
where $j\in I'\cup\{-\infty\}$, and
$
\chi_{w_i}^*p_2^*(h)=(p_2\chi_{w_i})^*(h)=\Sigma(s,i,w_i)^*(h)
$.
\end{proof}

\subsection{Copy and concentration operators}\label{Copy and concentration operators}
The results we have proved allow us to explain the existence of the operators of copy and concentration defined in~\cite{BTECBSV}
at least for a principal ideal domain $\k$ with invertible $2$. Let $\X_c(s)$ denote the image of the restriction
$H^\bullet(\BS_c(s),\k)\to H^\bullet(\Gamma(s),\k)$. We suppose that $s=(s_1,\ldots,s_n)$ is a nonempty
sequence.
For any $g\in\X_c(s')$, we consider the function
$\Delta g:\Gamma(s)\to S$ defined by $\Delta g(\gamma)=g(\gamma')$. We called this element
the {\it copy} of $g$, see~\cite[Section~4.2]{BTECBSV}. Let $h\in H^\bullet_T(\BS_c(s'),\k)$
be an element whose restriction to $\Gamma(s')$ is $g$. Then it is easy to understand that
$\Delta g$ is the restriction of $\tr^*(h)$ to $\Gamma(s)$, where $\tr:\BS_c(s)\to\BS_c(s')$
is the truncation map defined at the beginning of Section~\ref{Embeddings of Borel constructions}.
Indeed, let $\gamma\in\Gamma(s)$ and $j_\gamma:E^\infty/K\to \BS_c(s)\times_KE^\infty$
and $j_{\gamma'}:E^\infty/K\to \BS_c(s')\times_KE^\infty$ be the maps as
in Section~\ref{Equivariant_cohomology_of_a_point}. We get
$$
\tr^*(h)(\gamma)=j_\gamma^*\tr^*(h)=(\tr j_\gamma)^*(h)=j_{\gamma'}^*(h)=g(\gamma')=\Delta g(\gamma).
$$

On the other hand, for any $t\in\{1,s_n\}$, we consider the function $\nabla_tg$ called the {\it concentration}
of $g$ at $t$, see~Section~\cite[Section~4.2]{BTECBSV}. It is defined by
$$
\nabla_tg(\gamma)=
\left\{
\begin{array}{ll}
\gamma^n(-\alpha_n)g(\gamma')&\text{ if }\gamma_n=t;\\[6pt]
0&\text{ otherwise},
\end{array}
\right.
$$
where $\alpha_n$ is a positive root such that $s_n=s_{\alpha_n}$.
Let $c\in S$. Then by~(\ref{eq:pt:2}), the restriction of $\Sigma(s,i,1)^*(c)$ to $\gamma\in\Gamma(s)$ is given by
$$
\Sigma(s,i,1)^*(c)(\gamma)=j_\gamma^*\,\Sigma(s,i,1)^*(c)=(\Sigma(s,i,1)j_\gamma)^*(c)=(\rho_{(\gamma^i)^{-1}}/K)^*(c)
=\gamma^ic.
$$
Hence we get
$$
\nabla_tg=-\dfrac{\Sigma(s,n-1,1)^*(t\alpha_n)+\Sigma(s,n,1)^*(\alpha_n)}2\bigg|_{\Gamma(s)}\Delta g.
$$
Finally note that to get exactly the operators from~\cite{BTECBSV}, we need to consider only sequences $s$
of simple reflections and identify compactly defined and usual Bott-Samelson
varieties as well as $K$- and $T$-equivariant cohomologies.

\def\bs{\\[-4pt]}

\end{document}